\documentclass[12pt]{article}%

\usepackage{amscd,amsmath,amssymb,verbatim,amsfonts}
\usepackage[TS1,OT1,T1]{fontenc}

\usepackage{amsthm}
\usepackage{graphicx}%
\usepackage{hyperref}

\newtheorem{theorem}{Theorem}

\newtheorem{corollary}{Corollary}

\newtheorem{lemma}{Lemma}

\def\crg{\mathop{\rm cr}}
\def\rcr{\overline{\hbox{\rm cr}}}

\def\pscr{\widetilde{\hbox{\rm cr}}}

\def\oL#1{{\overline{#1}}}
\def\ese{{s}}
\def\eme{{r}}

\def\ebic#1{{E^{\text{\rm bic}}_{\le #1}}}
\def\emono#1{{E^{\text{\rm mono}}_{\le #1}}}
\def\edge#1{{\ell(#1)}}

\def\te{{t}}

\begin{document}

\title{\vspace{-.5in}On $\leq$$k$-edges, crossings, and halving lines \\of geometric
drawings of $K_{n}$ }
\author{Bernardo M. \'{A}brego$^{\text{1}}$, Mario
Cetina$^{\text{2}}$,\medskip \and Silvia
Fern\'{a}ndez-Merchant$^{\text{1}}$, Jes\'{u}s
Lea\~{n}os$^{\text{3}}$, Gelasio
Salazar$^{\text{4}}$\thanks{Supported by CONACYT Grant
106432}\bigskip
\\$^{\text{1}}${\small Department of Mathematics, California State University, Northridge. \hfill }
\\ {\footnotesize
\texttt{\hspace{0.3in}\{bernardo.abrego,silvia.fernandez\}@csun.edu}\hfill
}
\\$^{\text{2}}${\small Instituto Tecnol\'ogico de San Luis
  Potos\'{\i}.}
{\footnotesize \texttt{mario.cetina@itslp.edu.mx}\hfill }
\\$^{3}${\small Unidad Acad\'emica de Matem\'aticas, Universidad Aut\'{o}noma de
Zacatecas. \hfill }
\\ {\footnotesize
\texttt{\hspace{0.3in}jleanos@mate.reduaz.mx}\hfill}
\\$^{\text{4}}${\small Instituto de F\'isica, Universidad Aut\'{o}noma de San Luis
Potos\'{\i}. \hfill }
\\ {\footnotesize
\texttt{\hspace{0.3in}gsalazar@ifisica.uaslp.mx}\hfill}
}
\date{}
\maketitle \vspace{-0.4in}
\begin{abstract}
 Let $P$ be a set of points in general position in
the plane. Join all pairs of points in $P$ with straight line
segments. The number of segment-crossings in such a drawing, denoted
by $\crg(P)$, is the \emph{rectilinear crossing number} of $P$. A
\emph{halving line} of $P$ is a line passing though two points of
$P$ that divides the rest of the points of $P$ in (almost) half. The
number of halving lines of $P$ is denoted by $h(P)$. Similarly, a
$k$\emph{-edge}, $0\leq k\leq n/2-1$, is a line passing through two
points of $P$ and leaving exactly $k$ points of $P$ on one side. The
number of $(\le k)$-edges of $P$ is denoted by $E_{\leq k}(P)  $.
Let $\rcr(n)$, $h( n)$, and $E_{\leq k}(n)  $ denote the minimum of
$\crg(P)$, the maximum of $h(P)$, and the minimum of $E_{\leq k}(P)
$, respectively, over all sets $P$ of $n$ points in general position
in the plane. We show that the previously best known lower bound on
$E_{\leq k}(n)$ is tight for $k<\lceil ( 4n-2) /9\rceil $ and
improve it for all $k\geq \lceil ( 4n-2) /9 \rceil $. This in turn
improves the lower bound on $\rcr(n)$ from $0.37968\binom{n}
{4}+\Theta(n^{3})$ to $\frac{277}{729}\binom{n}{4}+\Theta(n^{3})\geq
0.37997\binom{n}{4}+\Theta(n^{3})$. We also give the exact values of
$\rcr(n)$ and $h(n)  $ for all $n\leq27$. Exact values were known
only for $n\leq18$ and odd $n\leq21$ for the crossing number, and
for $n\leq14$ and odd $n\leq21$ for halving lines.



\end{abstract}

\section{Introduction}

We consider three important well-known problems in Combinatorial
Geometry: the rectilinear crossing number, the maximum number of
halving lines, and the minimum number of $(\leq k) $-edges of
complete geometric graphs on $n$ vertices. All point sets in this
paper are in the plane, finite, and in general position.

Let $P$ be a finite set of points in general position in the plane.
The \emph{rectilinear crossing number }of $P$, denoted by $\crg(P)$,
is the number of crossings obtained when all straight line segments
joining pairs of points in $P$ are drawn. (A \emph{crossing }is the
intersection of two segments in their interior.) The
\emph{rectilinear crossing number }of $n$ is the minimum number of
crossings determined by any set of $n$ points, i.e., $\rcr(n)=\min\{
\crg(P):\vert P\vert =n\}$. The problem of determining $\rcr(n)$ for
each $n$ was posed by Erd\H{o}s and Guy in the early seventies
\cite{EG},\cite{G1}. This is equivalent to finding the minimum
number of convex quadrilaterals determined by $n$ points, as every
pair of crossing segments bijectively corresponds to the diagonals
of a convex quadrilateral.

A \emph{halving line} of $P$ is a line passing through two points of
$P$ and dividing the rest in almost half. So when $P$ has $n$ points
and $n$ is even, a halving line of $P$ leaves $n/2-1$ points of $P$
on each side; whereas when $n$ is odd, a halving line leaves $( n-3)
/2$ points on one side and $(n-1)/2$ on the other. The number of
halving lines of $P$ is denoted by $h(P) $. Generalizing a halving
line, a $k$\emph{-edge} of $P$, with $0\leq k\leq n/2-1$, is a line
through two points of $P$ leaving exactly $k$ points on one side.
The number of $k$-edges of $P$ is denoted by $E_{k}( P) $. Since a
halving line is a $( \lfloor n/2\rfloor -1) $-edge, then $E_{\lfloor
n/2\rfloor -1}( P) =h( P) $. Similarly, for $0\leq k \leq n/2-1$,
$E_{\leq k}( P) $ and $E_{\geq k}( P)  $ denote the number of $(
\leq k) $-edges and $( \geq k) $-edges of $P$, respectively. That
is, $E_{\leq k}( P) =\sum_{j=0}^{k}E_{j}( P)  $ and $E_{\geq k}( P)
=\sum_{j=k}^{\lfloor n/2\rfloor -1}E_{j}(  P)
=\binom{n}{2}-\sum_{j=0}^{k-1} E_{j}(  P) $. Let $h(n)$ and $E_{\leq
k}( n) $ be the maximum of $h(P)$ and the minimum of $E_{\leq k}( P)
$, respectively, over all sets $P$ of $n$ points. A concept closely
related to $k$-edges is that of \emph{$k$-sets}; a $k$-set of $P$ is
a set $Q$ that can be separated from $P \setminus Q$ with a straight
line. Rotating this separating line clockwise until it hits a point
on each side yields a $(k-1)$-edge, and it turns out that this
association is bijective. Thus the number of $k$-sets of $P$ is
equal to the number of $(k-1)$-edges of $P$. As a consequence, any
of the results obtained here for $k$-edges can be directly
translated into equivalent results for $(k+1)$-sets. Erd\H{o}s,
Lov\'{a}sz, Simmons, and Straus \cite{ELSS}, \cite{L} first
introduced the concepts of halving lines, $k$-sets, and $k$-edges.

Since the introduction of these parameters back in the early 1970s,
the determination (or estimation) of $\rcr(n)$, $h(n)$, and $E_{\le
k}(n)$ have become classical problems in combinatorial geometry.
General bounds are known but exact values have only been found for
small $n$. The best known general bounds for the halving lines are
$\Omega(ne^{c\sqrt{\log n}})\leq h(n)\leq O( n^{4/3})  $, due to
T\'{o}th \cite{T} and Dey \cite{D}, respectively. The previously
best asymptotic bounds for the crossing number are
\begin{equation}
0.3792\binom{n}{4}+\Theta(n^{3})\leq\rcr\left(  n\right)
\leq0.380488\binom{n}{4}+\Theta(n^{3}). \label{boundscr}
\end{equation}
The lower bound is due to Aichholzer et al. \cite{AGOR} and it
follows from Inequality (\ref{lower}) as we indicate below. The
upper bound follows from a recursive construction devised by
\'{A}brego and Fern\'{a}ndez-Merchant \cite{AF2} using the a
suitable initial construction found by the authors in \cite{ACFLS}.
The best lower bound for the minimum number of $( \leq k)  $-edges
is
\begin{equation}
E_{\leq k}\left(  n\right)  \geq3\binom{k+2}{2}+3\binom{k+2-\lfloor
n/3\rfloor}{2}-\max\left\{  0,(k+1-\lfloor n/3\rfloor)(n-3\lfloor
n/3\rfloor)\right\}  \text{,} \label{lower}
\end{equation}
due to Aichholzer et al. \cite{AGOR}. Further references and related problems
can be found in \cite{BMP}.

The last two problems are naturally related, and their connection to
the first problem is shown by the following identity, independently
proved by L\'{o}vasz et al. \cite{LVWW} and \'{A}brego and
Fern\'{a}ndez-Merchant \cite{AF}. For any set $P$ of $n$ points,

\begin{align}
\crg(P)  &  =3\binom{n}{4}-
{\displaystyle\sum\limits_{k=0}^{\left\lfloor n/2\right\rfloor -1}}
k\left(  n-k-2\right)  E_{k}\left(  P\right)  ,\text{ or
equivalently}
\nonumber\\
\crg(P)  &  = {\displaystyle\sum\limits_{k=0}^{\left\lfloor
n/2\right\rfloor -2}} \left(  n-2k-3\right)  E_{\leq k}\left(
P\right) -\frac{3}{4}\binom{n} {3}+\left(  1+\left(  -1\right)
^{n+1}\right) \frac{1}{8}\binom{n}{2}. \label{crossingsvsksets}
\end{align}
Hence, lower bounds on $E_{\leq k}(  n)  $ give lower bounds on
$\rcr(n)$.
\bigskip

The majority of our results (all non-constructive parts) are proved
in the more general context of generalized configurations of points,
where the points in $P$ are joined by pseudosegments rather than
straight line segments. Goodman and Pollack \cite{GP80} established
a correspondence between the set of generalized configurations of
points and what they called \emph{allowable sequences.} In Section
\ref{allowableseq}, we define allowable sequences, introduce the
necessary notation to state the three problems above in the context
of allowable sequences, and include a summary of results for these
problems in both, the geometric and the allowable sequence context.

\begin{table}[h]
\begin{center}%
\begin{tabular}
[c]{r|rrrrrrrrrr}%
$n$ & $14$ & $16$ & $18$ & $20$ & $22$ & $23$ & $24$ & $25$ & $26$ &
$27$\\\hline
\multicolumn{1}{l|}{$\overset{}{%
\begin{tabular}
[c]{l}%
$h(n)=\widetilde{h}(n)  \smallskip$%
\end{tabular}
\ \ \ \ \ }$} & $22^{\ast}$ & $27$ & $33$ & $38$ & $44$ & $75$ &
$51$ & $85$ &
$57$ & $96$\\
\multicolumn{1}{l|}{%
\begin{tabular}
[c]{l}%
$\rcr(n)=\pscr(n) $%
\end{tabular}
} & $324^{\ast}$ & $603^{\ast}$ & $1029^{\ast}$ & $1657$ & $2528$ &
$3077$ &
$3699$ & $4430$ & $5250$ & $6180$%
\end{tabular}
\end{center}
\caption{New exact values. The $^{\ast}$ values were only known in
the
rectilinear case.}%
\label{newvalues}%
\end{table}

The main result in this paper is Theorem \ref{main} in\ Section
\ref{proof central theorem}, which bounds $E_{\geq k}(P)$ by a
function of $E_{k-1}(P)$. This result has the following important
consequences.

\begin{enumerate}
\item In Section \ref{halving lines}, we find exact values of $\rcr (n)  $
and $h(n)$ for $n\leq27$. Exact values were only known for $n\leq18$
and odd $n\leq21$ in the case of $\rcr (n)$, and for $n\leq14$ and
odd $n\leq21$ in the case of $h(n)$. (See Table \ref{newvalues}.) We
also show that the same values are achieved for the more general
case of the pseudolinear crossing number $\pscr (n)$ and the maximum
number of halving pseudolines $\widetilde{h}( n)  $. (See Section
\ref{allowableseq} for the definitions.)

\item Theorem \ref{recursive} in Section \ref{lower bound k-sets} improves the lower
bound in Inequality (\ref{lower}) for $k\geq\left\lceil
(4n-11)/9\right\rceil $. It gives a recursive lower bound whose
asymptotic value is given by
\[
E_{\leq k}(  n)  \geq\binom{n}{2}-\frac{1}{9}\sqrt{1-\frac{2k+2}
{n}}(  5n^{2}+19n-31)  ,
\]
as shown in Corollary \ref{explicit}.

\item Theorem \ref{th: crossing} in Section \ref{lower crossings} improves the lower bound in Inequality (\ref{boundscr}) to
\[
\rcr(n) \geq\frac{277}{729}\binom{n}{4}+\Theta\left( n^{3}\right)
\geq0.37997\binom{n}{4}+\Theta\left(  n^{3}\right)  .
\]

\end{enumerate}

In Section \ref{constructions}, and to complement item 2 above, we
show that Inequality (\ref{lower}) is tight for $k<\left\lceil
(4n-11)/9\right\rceil $. More precisely, we construct sets of points
simultaneously achieving equality in Inequality (\ref{lower}) for
all $k<\left\lceil (4n-11)/9\right\rceil $.

Several results of this paper appeared (without proofs) in the
conference proceedings of LAGOS'07 \cite{AFLS2, AFLS3}.

\section{\label{allowableseq}Allowable sequences and generalized
configurations of points}

Any set $P$ of $n$ points in the plane can be encoded by a sequence
of permutations of the set $[  n]  =\{ 1,2,...,n\} $ as follows.
Consider a directed line $l$. Orthogonally project $P$ onto $l$ and
label the points of $P$ from $1$ to $n$ according to their order in
$l$. In this order, the identity permutation $( 1,2,...,n)  ,$ is
the first permutation of our sequence. Note that $l$ can be chosen
so that none of the projections overlap. Continuously rotate $l$
counterclockwise. The order of the projections of $P$ onto $l$
changes every time two projections overlap, that is, every time a
line through two points of $P$ becomes perpendicular to $l$. Each
time this happens, a new permutation is recorded as part of our
sequence. After a $180^{\circ}$-rotation of $l$ we obtain a sequence
of $\binom{n}{2}+1$ permutations such that the first permutation $(
1,2,...,n)  $ is the identity,
 the last permutation $(  n,n-1,...,2,1)$ is the reverse of the
identity, any two consecutive permutations differ by a transposition
of adjacent elements, and any pair of points (labels $1,...,n$)
transpose exactly once. This sequence is known as a
\emph{halfperiod} \emph{of the circular sequence }associated to $P$.
The \emph{circular sequence} of $P$ is then a doubly infinite
sequence of permutations obtained by rotating $l$ indefinitely in
both directions.

As an abstract generalization of a circular sequence, a \emph{simple
allowable sequence }on $[n]  $ is a doubly infinite sequence $\Pi=(
...,\pi_{-1},\pi_{0},\pi_{1},...)  $ of permutations of $[ n] $,
such that any two consecutive permutations $\pi_{i}$ and $\pi_{i+1}$
differ by a transposition $\tau(\pi_{i})$ of neighboring elements,
and such that for every $j$, $\pi_{j}$ is the reverse permutation of
$\pi_{j+\binom{n}{2}}$. A \emph{halfperiod} of $\Pi$ is a sequence
of $\binom{n}{2}+1$ consecutive permutations of $[ n] $. As before,
any halfperiod of $\Pi$ uniquely determines $\Pi$ and all properties
for halfperiods mentioned above still hold. Moreover, the halfperiod
$\pi=( \pi_{i},\pi_{i+1},...,\pi_{i+\binom{n}{2}})  $ is completely
determined by the transpositions
$\tau(\pi_{i}),\tau(\pi_{i+1}),\ldots,\tau(\pi_{i+\binom{n}{2}-1}).$
Note that the sequence $(\ldots,\tau(\pi_{-1}), \tau(\pi_0),
\tau(\pi_1)\ldots)$ is $\tbinom{n}{2}$-periodic. Thus we
indistinctly refer to $\pi$ as a sequence of permutations or as a
sequence of (suitable) transpositions. Allowable sequences that are
the circular sequence of a set of points are called
\emph{stretchable}.

A \emph{pseudoline} is a curve in $\mathbb{P}^{2}$, the projective plane,
whose removal does not disconnect $\mathbb{P}^{2}$. Alternatively, a
pseudoline is a simple curve in the plane that extends infinitely in both
directions. A \emph{simple generalized configuration} \emph{of points}
consists of a set of $\binom{n}{2}$ pseudolines and $n$ points in the plane
such that each pseudoline passes through exactly two points, and any two pseudolines
intersect exactly once.

Circular and allowable sequences were first introduced by Goodman and Pollack
\cite{GP80}. They proved that not every allowable sequence is stretchable and
established a correspondence between allowable sequences\emph{ }and
generalized configurations of points.

The three problems at hand can be extended to generalized
configurations of points, or equivalently, to simple allowable
sequences. In this new setting, a transposition of two points in
positions $k$ and $k+1$, or $n-k$ and $n-k+1$ in a simple allowable
sequence $\Pi$ corresponds to a $(  k-1) $-edge. We say that such
transposition is a $k$-transposition, or respectively, a $( n-k)
$-transposition, and if $1\leq k\leq n/2$ all these transpositions
are called $k$\emph{-critical}. Therefore $E_{k}( \Pi)  ,$ $E_{\leq
k}(  \Pi)  ,$ and $E_{\geq k}( \Pi)  $ correspond to the number of
$( k+1) $-critical, $(  \leq k+1)  $-critical, and $( \geq k+1)
$-critical transpositions in any halfperiod of $\Pi$. A halving line
of $\Pi$ is a $\lfloor n/2\rfloor $-transposition, and thus $h( \Pi)
=E_{\lfloor n/2\rfloor -1}( \Pi) $. Identity
(\ref{crossingsvsksets}), which relates the number of $k$-edges to
the crossing number, was originally proved for allowable sequences.
In this setting, a \emph{pseudosegment} is the segment of a
pseudoline joining two points in a generalized configuration of
points, and $\crg( \Pi)  $ is the number of pseudosegment-crossings
in the generalized configuration of points that corresponds to the
allowable sequence $\Pi$. All these definitions and functions
coincide with their original counterparts for $P$ when $\Pi$ is the
circular sequence of $P$. However, when $\rcr(n),$ $h( n) ,$ and
$E_{\leq k}( n)  $ are minimized or maximized over all allowable
sequences on $[ n] $ rather than over all sets of $n$ points, the
corresponding quantities may change and therefore we use the
notation $\pscr(n),\widetilde{h}(n),$ and $\widetilde{E}_{\leq k}(
n) $. Because $n$-point sets correspond to the stretchable simple
allowable sequences on $[n]$, it follows that $\pscr(n)\leq\rcr(n)$,
$\widetilde{h}(n)\geq h(n)$, and $\widetilde {E}_{\leq k}( n) \leq
E_{\leq k}( n)  .$ Tamaki and Tokuyama \cite{TT} extended Dey's
upper bound for allowable sequences to $\widetilde{h}(n)=O(n^{4/3})$
. \'{A}brego et al. \cite{ABFLS2} proved that the lower bound for
$E_{\leq k}( n)$ in Inequality (\ref{lower}) is also a lower bound
on $\widetilde{E}_{\leq k}( n) $. They used this bound to extend
(and even slightly improve) the corresponding lower bound on
$\rcr(n)$ to $\pscr(n)$.

Our main result, Theorem \ref{main} in Section \ref{proof central
theorem}, concentrates on the central behavior of allowable
sequences. We bound $E_{\geq k}(  \Pi)  $ by a function of $E_{k-1}(
\Pi)  $. As a consequence, we improve (or match) the upper bounds on
$\widetilde{h}(n)$ for $n\leq27,$ and thus the lower bounds on
$\pscr(n)$ in the same range. This is enough to match the
corresponding best known geometric constructions \cite{A} for $h( n)
$ and $\rcr(n)$. This shows that for all $n\leq27,$
$\widetilde{h}(n)=h(  n)  $ and $\pscr(n)=\rcr(n)$ whose exact
values are summarized in Table \ref{newvalues}.

\section{\label{proof central theorem}The Central Theorem}

In this section, we present our main theorem. Given a halfperiod
$\pi=( \pi_{0},\pi_{1},\pi_{2},...,\pi_{\binom{n}{2}})  $ of an
allowable sequence and an integer $1\leq k < n/2$, the
$k$\emph{-center} of the permutation $\pi_{j}$, denoted by
$C(k,\pi_{j})  $, is the set of elements in the middle $n-2k$
positions of $\pi_{j}$. Let $L_{0},C_{0},$ and $R_{0}$ be the set of
elements in the first $k$, middle $n-2k$, and last $k$ positions,
respectively, of the permutation $\pi_{0}$. Define
\[
s\left(  k,\pi\right)  =\min\left\{  \left\vert C_{0}\cap C\left(  k,\pi
_{i}\right)  \right\vert :0\leq i\leq\binom{n}{2}\right\}  .
\]
Note that $s(  k,\pi)  \leq n-2k-1$ because at least one of the
$n-2k$ elements of $C_{0}$ must leave the $k$-center.

\begin{theorem}
\label{main}Let $\Pi$ be an allowable sequence on $[  n]  $ and
$\pi$ any halfperiod of $\Pi$. If $s=s(k,\pi)  $, then
\[
E_{\geq k}\left(  \Pi\right)  \leq\left(  n-2k-1\right)
E_{k-1}\left( \Pi\right)  -\frac{s}{2}\left(  E_{k-1}\left(
\Pi\right)-n+1 \right)  .
\]

\end{theorem}

\begin{proof} For presentation purposes, we divide this proof into subsections.\medskip

Let $\Pi$ be an allowable sequence on $[n]  $ and $\pi=
(\pi_{0},\pi_{1},\pi_{2},...,\pi_{\binom{n}{2}})  $ any halfperiod
of $\Pi$, $s=s(  k,\pi)  $, and $K=E_{k-1}( \pi) $.

Suppose that $\pi_{i_{1}},\pi_{i_{2}},...,\pi_{i_{K}}$ is the
subsequence of permutations in $\pi$ obtained when the $k$-critical
transpositions
$\tau(\pi_{i_{1}}),\tau(\pi_{i_{2}}),...,\tau(\pi_{i_{K}})$ of $\pi$
occur (in this order). For simplicity we write $\tau_j$ instead of
$\tau(\pi_{i_j})$. These permutations partition $\pi$ into $K+1$
parts $B_{0}( \pi) ,$ $B_{1}( \pi)  ,$ $B_{2}( \pi) ,$ $...,$
$B_{K}( \pi) $
called \emph{blocks}, where $B_{j}(  \pi)  =\{  \pi_{l}%
:i_{j}\leq l<i_{j+1}\}  $ for $1\leq j\leq K-1$, $B_{0}( \pi) =\{
\pi_{l}:0\leq l<i_{1}\}  $, and $B_{K}( \pi)  =\{ \pi_{l}:i_{K}\leq
l\leq\binom{n}{2}\} $. Denote by $p_{j}$ the point that enters the
$k$-center of $\pi_{i_{j}}$ with $\tau_{j}$. We say that a $( \geq
k+1) $-critical transposition in $B_{j}( \pi) ,1\leq j\leq K,$ is an
\emph{essential} transposition if it involves $p_{j}$ or if it
occurs before $\tau_{1}$, and a \emph{nonessential} transposition
otherwise.

\begin{figure}[hp]
\begin{center}
\includegraphics[
height=2.911in, width=5.0825in ] {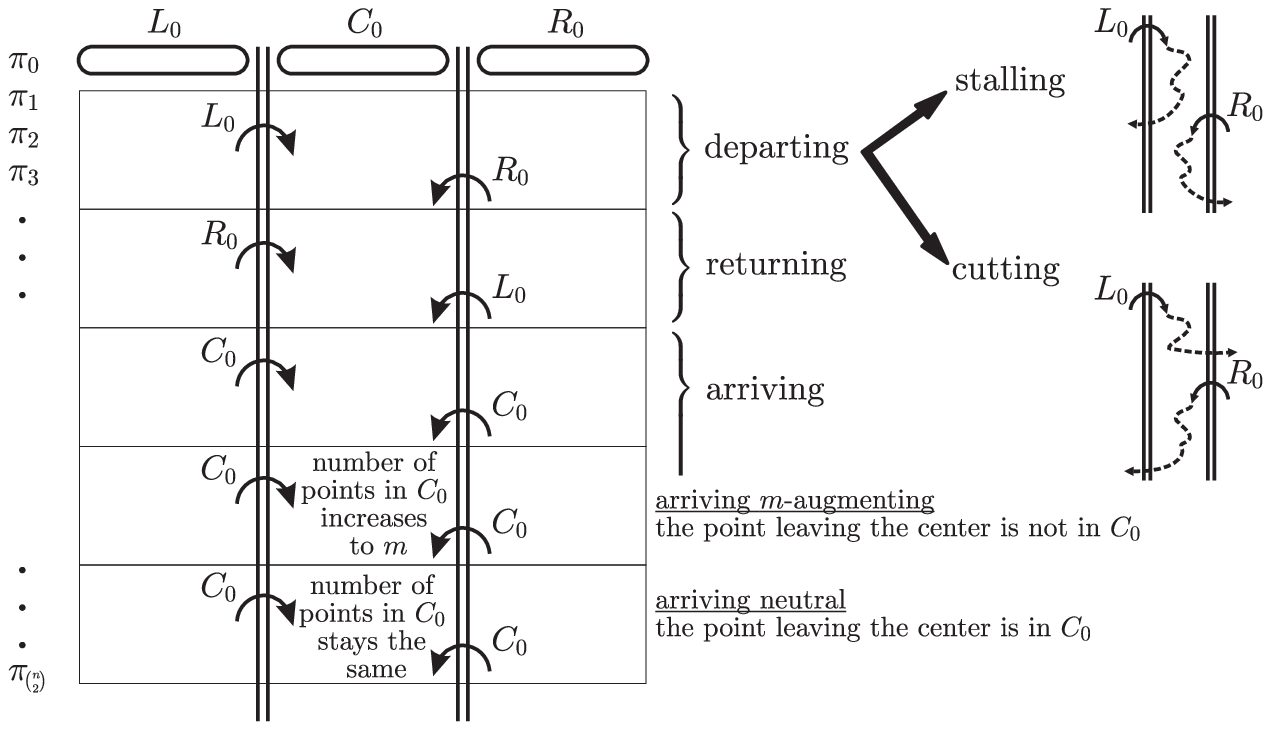}
\caption{Classification of essential $k$-critical transpositions.}
\label{fig:classification}
\end{center}
\end{figure}

\subsubsection*{Rearrangement of $\pi$}

We claim that, to bound $E_{\geq k}(  \Pi)  $, we can assume that
all $(  \geq k+1)  $-critical transpositions of $\pi$ are essential
transpositions. To show this, in case $\pi$ has nonessential
transpositions, we modify $\pi$ so that the obtained halfperiod
$\lambda$ satisfies $E_{j}(  \pi) =E_{j}( \lambda)  $ for all $j<k$,
and thus $E_{\geq k}( \pi) =E_{\geq k}( \lambda) $; and either
$\lambda$ has only essential transpositions or the last nonessential
transposition of $\lambda$ occurs in an earlier permutation than the
last nonessential transposition of $\pi$. Applying this procedure
enough times, we end with a halfperiod $\lambda$ all of whose $(
\geq k+1) $-critical transpositions are essential and such that
$E_{j}( \pi) =E_{j}( \lambda)  $ for all $j\leq k$, and thus
$E_{\geq k}( \pi)  =E_{\geq k}( \lambda) $.

This is how $\lambda$ is constructed. Suppose $B_{j}(  \pi)  $ is
the last block of $\pi$ that contains nonessential transpositions.
Define $\lambda$ as the halfperiod that coincides with $\pi$
everywhere except for the $(  \geq k+1) $-transpositions in $B_{j}(
\pi)  $. All nonessential transpositions in $B_{j}( \pi) $ take
place right before $\tau_{j}$ in $\lambda$, and right after
$\tau_{j}$ occurs, all essential transpositions in $B_{j}( \pi)  $
occur consecutively in $B_{j}( \lambda)  $ but probably in a
different order than in $B_{j}( \pi)  $, so that the final position
of $p_{j}$ is the same in $B_{j}(  \pi)  $ and $B_{j}( \lambda)  $.
Note that in fact the last permutations of the blocks $B_{j}( \pi) $
and $B_{j}( \lambda)  $ are equal.

\subsubsection*{Classification of $k$-critical transpositions}

From now on, we assume that $\pi$ only has essential transpositions.
We classify the $k$-critical transpositions as follows (see Figure
\ref{fig:classification}): $\tau_{j}$ is an \emph{arriving}
transposition if $p_{j}\in C_{0}$. An arriving transposition is
$m$\emph{-augmenting} if it increments the number of elements in
$C_{0}$ in the $k$-center from $m-1$ to $m$, and it is \emph{neutral
}otherwise. We say that $\tau_{j}$ is a \emph{returning}
transposition if it is a $k$-transposition and $p_{j}\in R_{0}$, or
if it is an $(n-k)$-transposition and $p_{j}\in
L_{0}$. That is, $p_{i}$ is \textquotedblleft getting back\textquotedblright%
\ to its starting region. Similarly, $\tau_{j}$ is a
\emph{departing} transposition if it is a $k$-transposition and
$p_{j}\in L_{0}$, or if it is an $(n-k)$-transposition and $p_{j}\in
R_{0}$. That is, $p_{j}$ is \textquotedblleft getting
away\textquotedblright\ from its original region. We say that a
departing transposition $\tau_{j}$ is a \emph{cutting
}transposition, if $\tau_{j}$ is a $k$-transposition and the next
$k$-critical transposition that involves $p_{j}$ is an
$(n-k)$-transposition; or if $\tau_{i}$ is an $(n-k)$-transposition
and the next $k$-critical transposition that involves $p_{j}$ is a
$k$-transposition. All other departing transpositions are called
\emph{stalling}.

Finally, we define the \emph{weight} of a $k$-critical transposition
$\tau _{j}$, denoted by $w(\tau_{j})$, as the number of $( \geq k+1)
$-critical transpositions in $B_{j}( \pi)  $ that are not between
two elements of $C_{0}$. Transpositions with weight at most
$n-2k-1-s$ are called \emph{light}. All other transpositions are
\emph{heavy}.

Let $A,N,R,C,S_{\text{light}}$, and $S_{\text{heavy}}$ be the number of
augmenting, neutral, returning, cutting, light stalling, and heavy stalling
transpositions, respectively. Then $K=A+N+R+C+S_{\text{light}}+S_{\text{heavy}%
}$.

\subsubsection*{Bounding $E_{\geq k}(  \Pi)  $}

Observe that the $k$-center of all permutations in $B_{0}( \pi)  $
remains unchanged. It follows that all $(  \geq k+1) $-critical
transpositions of $B_{0}( \pi)  $ are between elements of $C_{0}$.
Thus $\sum_{j=1}^{K}w(  \tau_{j})  $ counts all $(  \geq k+1)
$-critical transpositions except those between two elements of
$C_{0}$. There are $\binom{n-2k}{2}$ transpositions between elements
of $C_{0}$, but each neutral transposition corresponds to a
$k$-critical (not $( \geq k+1)  $-critical) transposition between
two elements of $C_{0}$. Thus
\begin{equation}
E_{\geq k}(  \Pi)  \leq\binom{n-2k}{2}-N+\sum\limits_{j=1}^{K}w(
\tau_{j}). \label{degrees inequality}
\end{equation}

\subsubsection*{Bounds for the weight of a $k$-critical transposition}

We bound the weight of a transposition depending on its class
(departing, returning, etc.), as well as the number of
transpositions within a class, if necessary. For $j \geq 1$ all $(
\geq k+1)  $-critical transpositions in $B_{j}(  \pi)  $ involve
$p_{j}$ and thus $w(  \tau _{j})  \leq n-2k-1.$ However, since the
weight of $\tau_{j}$ does not count transpositions between two
elements of $C_{0}$, and there are always at least $s$ elements of
$C_{0}$ in the $k$-center, then $w( \tau _{j})  \leq n-2k-s$
whenever $\tau_{j}$ is arriving (because $p_{j}\in C_{0}$).
Moreover, if $\tau_{j}$ is $m$-augmenting, then $w(  \tau _{j}) \leq
n-2k-m$. If $\tau_{j}$ is a returning transposition, then $p_{j}$
has already been transposed with all the elements of $C_{0}$ that
are in the $k$-center of $\pi_{i_{j}}$. Since there are at least $s$
such elements, then $w(  \tau_{j})  \leq n-2k-1-s$. Summarizing,
\begin{equation}
w\left(  \tau_{j}\right)  \leq\left\{
\begin{array}
{ll}%
n-2k-1&\text{for all }\tau_{j}\text{,}\\
n-2k-s,&\text{if }\tau_{j}\text{ is neutral,}\\
n-2k-m,&\text{if }\tau_{j}\text{ is }m\text{-augmenting,}\\
n-2k-1-s,&\text{if }\tau_{j}\text{ is light stalling or returning.}%
\end{array}
\right.  \label{bounds on weights}%
\end{equation}

\subsubsection*{Bounding $C$}

We bound the number of cutting transpositions. Since the first
(last) $k$ elements of $\pi_{0}$ are the last (first) elements of
$\pi_{\binom{n}{2}}$, then the $2k$ elements not in $C_{0}$ must
participate in at least one cutting transposition. That is,
$C\geq2k$. Note that, if $p\notin C_{0}$ participates in $c\geq2$
cutting transpositions, then there must be at least $c-1$ returning
transpositions of $p$. In other words, there must be at least
$C-2k\geq0$ returning transpositions. There are $C$ cutting
transpositions and at least $n-2k-s$ arriving transpositions (at
least one $m$-augmenting arriving transposition for each $s+1\leq
m\leq n-2k$). Then $K-C-(  n-2k-s)  $ counts all other $k$-critical
transpositions, including in particular all returning
transpositions. Thus $K-C-(
n-2k-s)  \geq C-2k$, that is,%
\begin{equation}
2C \leq 4k+K-n+s. \label{bounding crossing transp}%
\end{equation}

\subsubsection*{Augmenting and heavy stalling transpositions}

We keep track of the augmenting and heavy stalling transpositions
together. To do this, we consider the bipartite graph $G$ whose
vertices are the augmenting and the heavy stalling transpositions.
The augmenting transposition $\tau_{l}$ is adjacent in $G$ to the
heavy stalling transposition $\tau_{j}$ if $j<l$, $p_{j}$ is in the
$k$-center of all permutations in blocks $B_{j}$ to $B_{l}$, one
transposition from $\tau_{j}$ and $\tau_{l}$ is a $k$-transposition
and the other is an $(  n-k)  $-transposition, and $p_{l}$ does not
swap with $p_{j}$ in $B_{l}(  \pi)  $. We bound the degree of a
vertex in $G$.

Let $\tau_{j}$ be a heavy stalling transposition. If $p_{j}\in
L_{0}$ (the case $p_{j}\in R_{0}$ is equivalent), then $\tau_{j}$ is
a $k$-transposition. Because $p_{j}$ moves to the right exactly $w(
\tau_{j})>n-2k-1-s$ positions within $B_{j}( \pi)  $, it follows
that the $k$-center right before $\tau_{j+1}$ occurs (i.e., the
$k$-center of $\pi_{i_{j+1}-1}$) has at most $n-2k-1-w( \tau_{j})
<s$ points of $C_{0}$ to the right of $p_{j}$. Also, since
$\tau_{j}$ is stalling, the next time that $p_{j}$ leaves the
$k$-center is by a $k$-transposition $\tau_{j+a}$. This means that
the $k$-center right before $\tau_{j+a}$ occurs (i.e., the
$k$-center of $\pi_{i_{j+a}-1}$) has at least $s$ points of $C_{0}$
to the right of $p_{j}$. Thus, between $\tau_{j}$ and $\tau_{j+a}$
there must be at least $s-( n-2k-1-w( \tau_{j})  ) $ arriving $(
n-k) $-transpositions $\tau_{l}$ such that $p_{l}$ remains to the
right of $p_{j}$ in $B_{l}(  \pi)  $, i.e., $p_{l}$ does not swap
with $p_{j}$ in $B_{l}(  \pi)  $. These transpositions are adjacent
to $\tau_{j}$ and thus the degree of $\tau_{j}$ in $G$ is at least
$w( \tau_{j}) -(  n-2k-1-s)  $. Hence,
\[
\left\vert E(G)  \right\vert \geq\sum\limits_{\tau_{j}\text{ heavy
stalling}}\left(  w\left(  \tau_{j}\right)  -\left( n-2k-1-s\right)
\right)  ,
\]
where $E(G)$ is the set of edges of $G$.

Let $\tau_{l}$ be an $m$-augmenting transposition. Since $p_{l}\in
C_{0}$, and weights do not count transpositions between two elements
of $C_{0}$, then at most $n-2k-m-w(  \tau_{l})  $ points in
$L_{0}\cup R_{0}$ do not swap with $p_{l}$ in $B_{l}( \pi)  $. Only
these points are possible $p_{j}$s such that $\tau_{j}$ is adjacent
to $\tau_{l}$. Thus the degree of $\tau_{l}$ in $G$ is at most
$n-2k-m-w(  \tau_{l})  \leq n-2k-1-s-w( \tau_{l}) $.

Note that there is at least one $m$-augmenting transposition for
each $s+1\leq m\leq n-2k$. This is because the $k$-center of at
least one permutation of $\pi$ contains exactly $s$ elements of
$C_{0}$ (by definition of $s$), and the $k$-center of
$\pi_{\binom{n}{2}}$ contains exactly $n-2k$ elements of $C_{0}$
(since it coincides with $C_{0}$). Then the number of elements in
the $k$-center must be eventually incremented from $s$ to $n-2k$.
For each $s+1\leq m\leq n-2k$, we use $n-2k-m-w(  \tau_{l})  $ to
bound the degree of \emph{one }$m$-augmenting transposition. For all
other augmenting transpositions we use the bound $n-2k-1-s-w(
\tau_{l})  $. Hence
\begin{align*}
\left\vert E\left(  G\right)  \right\vert  &
\leq\sum\limits_{\tau_{j}\text{ augmenting}}\left(  \left(
n-2k-1-s\right)  -w\left(  \tau_{j}\right)
\right)  -\sum\limits_{m=s+1}^{n-2k}\left(  m-s-1\right) \\
&  =\sum\limits_{\tau_{j}\text{ augmenting}}\left(  \left(
n-2k-1-s\right) -w\left(  \tau_{j}\right)  \right)
-\binom{n-2k-s}{2}.
\end{align*}
The previous two inequalities imply that
\begin{equation}
\sum\limits_{\tau_{j}\text{ augmenting}}w\left(  \tau_{j}\right)
+\sum\limits_{\tau_{j}\text{ heavy stalling}}w\left( \tau_{j}\right)
\leq\left(  n-2k-1-s\right)  \left( A+S_{\text{heavy}}\right)
-\binom {n-2k-s}{2}. \label{augmenting and heavy stall}
\end{equation}

\subsubsection*{Final calculations}

We use inequalities (\ref{bounds on weights}) and (\ref{augmenting
and heavy stall}) to bound $\sum_{i=1}^{K}w(  \tau_{i})  -N$.
\begin{align*}
\sum\limits_{j=1}^{K}w\left(  \tau_{j}\right)  -N &
=\sum\limits_{\tau _{j}\text{ cutting}}w\left(  \tau_{j}\right)
+\sum\limits_{\tau_{j}\text{ augmenting}}w\left(  \tau_{j}\right)
+\sum\limits_{\tau_{j}\text{ heavy
stalling}}w\left(  \tau_{j}\right)  \\
&  +\sum\limits_{\tau_{j}\text{ light stalling}}w\left(
\tau_{j}\right) +\sum\limits_{\tau_{j}\text{ returning}}w\left(
\tau_{j}\right)
+\sum\limits_{\tau_{j}\text{ neutral}}w\left(  \tau_{j}\right)  -N\\
&  \leq\left(  n-2k-1\right)  C+\left(  n-2k-1-s\right)  \left(
A+S_{\text{heavy}}\right)  -\binom{n-2k-s}{2}\\
&  +\left(  n-2k-1-s\right)  \left(  S_{\text{light}}+R\right)  +\left(
n-2k-s\right)  N-N\\
&  \leq sC+\left(  n-2k-1-s\right)  K-\binom{n-2k-s}{2}.
\end{align*}
By Inequality (\ref{degrees inequality}),
\begin{align*}
E_{\geq k}\left(  \Pi\right)   &  \leq\binom{n-2k}{2}-\binom{n-2k-s}
{2}+sC+\left(  n-2k-1-s\right)  K\\
&  =\left(  n-2k-1\right)  K-\frac{s}{2}\left(  2K-2n+4k+1+s-2C\right)  .
\end{align*}
Finally, by Inequality (\ref{bounding crossing transp}),
\[
E_{\geq k}\left(  \Pi\right)  \leq\left(  n-2k-1\right)
K-\frac{s}{2}\left( K-n+1  \right). \qedhere
\]
\end{proof}

\section{\label{halving lines}New exact values for $n\leq27$}

In this section, we give exact values of $h( n) $ and $\widetilde{%
h}( n) $ for $n\leq27$. We start by stating a relaxed version of
Theorem \ref{main}, which we use in the special case when $k=\lfloor
n/2\rfloor -1$.

\begin{corollary}
\label{coro: maxs}Let $\Pi $ be a simple allowable sequence on $[ n]
$ and $\pi $ any halfperiod of $\Pi $. If $s=s( k,\pi ) $, then
\begin{equation*}
E_{\geq k}\left( \Pi \right) \leq \left( n-2k-1\right) E_{k-1}\left(
\Pi \right) +\binom{s}{2}\leq \left( n-2k-1\right) E_{k-1}\left( \Pi
\right) + \binom{n-2k-1}{2}.
\end{equation*}
\end{corollary}

\begin{proof} There are at least $n-2k-s$ elements of $C_{0}$ that leave the
$k$-center, so there are at least $n-2k-s$ arriving transpositions.
In addition, there are at least $ 2k$ departing transpositions, one
per element not in $C_{0}$. It follows that $E_{k-1} ( \Pi) \geq2k+(
n-2k-s) =n-s$. The first inequality now follows directly from
Theorem \ref{main}. Finally, $ s\leq n-2k-1$ for all halfperiods of
$\Pi$ which yields the second inequality. Another consequence is
that $E_{k-1} ( \Pi) \geq n-s \geq 2k+1$, which is in fact the
minimum possible value of $E_{k-1}$ (cf. \cite{LVWW}).
\end{proof}

The previous corollary implies the following result for halving
lines.

\begin{corollary}
\label{coro: halving}If $\Pi $ is a simple allowable sequence on $[ n%
] $ and $n\geq 8$, then%
\begin{equation*}
h\left( \Pi \right) \leq \left\{
\begin{array}{ll}
\left\lfloor \frac{1}{24}n(n+30)-3\right\rfloor &\text{ if }n\text{
is
even,}\vspace{0.1in} \\
\left\lfloor \frac{1}{18}(n-3)(n+45)+\frac{1}{9}\right\rfloor
&\text{ if }n \text{ is odd.}
\end{array}%
\right.
\end{equation*}
\end{corollary}

\begin{proof} If $k=\lfloor n/2\rfloor -1$ on Corollary \ref{coro: maxs},
then $ E_{\geq \lfloor n/2\rfloor -1}(\Pi )=h(\Pi )$ and thus $h(\Pi
)\leq (n-2\lfloor n/2\rfloor +1)E_{\geq \lfloor n/2\rfloor -2}(\Pi
)+\binom{ n-2\lfloor n/2\rfloor +1}{2}$, that is,
\begin{equation*}
h\left( \Pi \right) \leq \left\{
\begin{array}{ll}
E_{n/2-2}\left( \Pi \right) &\text{ if }n\text{ is even,} \vspace{0.1in}\\
2E_{\left( n-1\right) /2-2}\left( \Pi \right) +1 &\text{ if }n\text{
is odd.}
\end{array}
\right.
\end{equation*}
Moreover, because $E_{\leq \lfloor n/2\rfloor -3}(\Pi )+E_{\lfloor
n/2\rfloor -2}(\Pi )+h(\Pi )=\binom{n}{2}$, it follows that
\begin{equation*}
h\left( \Pi \right) \leq \left\{
\begin{array}{ll}
\left\lfloor \frac{1}{2}\binom{n}{2}-\frac{1}{2}E_{\leq n/2-3}\left(
\Pi
\right) \right\rfloor &\text{ if }n\text{ is even, } \vspace{0.1in}\\
\left\lfloor \frac{2}{3}\binom{n}{2}-\frac{2}{3}E_{\leq \left(
n-1\right) /2-3}\left( \Pi \right) +\frac{1}{3}\right\rfloor &\text{
if }n\text{ is odd.}
\end{array}%
\right.
\end{equation*}%
The bound in Inequality (\ref{lower}) is also valid in the more
general context of allowable sequences \cite{ABFLS2}. Using this
bound for $E_{\leq k}(\Pi )$ when $k=\lfloor n/2\rfloor -3$, and
considering all residue
classes of $n$ modulo 18 with $n\geq 8$, it follows that $\lfloor \frac{1}{2}%
\binom{n}{2}-\frac{1}{2}E_{\leq n/2-3}( \Pi ) \rfloor \leq
\lfloor n(n+30)/24-3\rfloor $ when $n$ is even, and $\lfloor \frac{2}{3}%
\binom{n}{2}-\frac{2}{3}E_{\leq ( n-1) /2-3}( \Pi ) +%
\frac{1}{3}\rfloor \leq \lfloor (n-3)(n+45)/18+1/9\rfloor $ when $n$
is odd.
\end{proof}

Because $h(n)\leq cn^{4/3}$, the inequality in Corollary \ref{coro:
halving} is only useful for small values of $n$. However, even with
the current best constant $c=(31287/8192)^{1/3}<1.5721$ \cite{AA,
PRTT}, our bound is better when $n$ is even in the range $8\leq
n\leq 184$.

The exact values of $h(n)$ were previously known only for even $n
\le 14$ or odd $n\le 21$ \cite{AA, BR}. The exact values of $\rcr
(n)$ were previously known only for even $n\leq 18$ or odd $n\leq
21$ \cite{AGOR}. The values in Table \ref{newvalues} correspond to
the upper bounds obtained by Corollary \ref{coro: halving} when $n$
is even, $14\leq n\leq 26$ or $n$ is odd, $23\leq n\leq 27$. We also
obtained new lower bounds for $\widetilde{ \hbox{\rm cr}}( n) $ in
this range of values of $n$. The identity $E_{\leq \lfloor
n/2\rfloor -2}( \Pi ) =\binom{n}{2} -h(\Pi )$ together with
Corollary \ref{coro: halving} give a new lower bound for $E_{\leq
\lfloor n/2\rfloor -2}( \Pi ) $. Using this bound for $k=\lfloor
n/2\rfloor -2$ and the bound in Inequality (\ref{lower}) for $k\leq
\lfloor n/2\rfloor -3$ in Identity (\ref{crossingsvsksets}) yields
the values in Table \ref{newvalues} for $\widetilde{\hbox{\rm cr}}(
n) $. For example, if $n=24$ then $E_{\leq 10}(\Pi
)=\binom{24}{2}-h(24)\geq 276-51=225$ and by Inequality (\ref
{lower}), the vector $(E_{\leq 0}(\Pi ),E_{\leq 1}(\Pi ),E_{\leq
2}(\Pi ),\ldots ,E_{\leq 9}(\Pi ))$ is bounded below entry-wise by $
(3,9,18,30,45,63,84,108,138,174)$, so Identity
(\ref{crossingsvsksets}) implies that $\widetilde{\hbox{\rm cr}}(
24) =\sum_{k=0}^{10}(21-2k)E_{\leq k}(\Pi
)-\frac{3}{4}\binom{24}{3}\geq 3699$.

All the bounds shown in Table \ref{newvalues} are attained by
Aichholzer's et al. constructions \cite{A}, and thus Table
\ref{newvalues} actually shows
the exact values of $\widetilde{h}( n) $, $h( n) $, $%
\widetilde{\hbox{\rm cr}}( n) $, and $\overline{\hbox{\rm cr}}%
( n) $ for $n$ in the specified range.

For $28\leq n\leq 33$, Table \ref{smallbounds} shows the new reduced
gap
between the lower and upper bounds of $h( n) $ and $\widetilde{h}%
( n) $.

\begin{table}[h]
\begin{center}%
\begin{tabular}
[c]{r|rrrrrr}%
$n$ & $28$ & $29$ & $30$ & $31$ & $32$ & $33$\\\hline
$\overset{}{%
\begin{tabular}
[c]{l}%
$h\left(  n\right)  \geq\smallskip$%
\end{tabular}
}$ & $63$ & $105$ & $69$ & $115$ & $73$ & $126$\\%
\begin{tabular}
[c]{l}%
$\widetilde{h}\left(  n\right)  \leq\smallskip$%
\end{tabular}
& $64$ & $107$ & $72$ & $118$ & $79$ & $130$\\%
\end{tabular}
\end{center}
\caption{Updated bounds for $28\leq n\leq33$}%
\label {smallbounds}
\end{table}

\section{\label{lower bound k-sets}New lower bound for the number of $(
\leq k)  $-edges}

In this section, we obtain a new lower bound for the number of
$\leq$$k$-edges. Our emphasis is on finding the best possible
asymptotic result as well as the best bounds that apply to the small
values of $n$ for which the exact value is unknown. Theorem
\ref{recursive} provides the exact result that can be applied to
small values of $n$, whereas Corollary \ref{explicit} is suitable
enough to give the best asymptotic behavior.

Let $m=\lceil(4n-11)/9\rceil$. For each $n$, define the following
recursive
sequence.%
\begin{align*}
u_{m-1}  &  =3\binom{m+1}{2}+3\binom{m+1-\lfloor n/3\rfloor}{2}-3\left(
m-\left\lfloor \frac{n}{3}\right\rfloor \right)  \left(  \frac{n}%
{3}-\left\lfloor \frac{n}{3}\right\rfloor \right)  \text{ and}\\
u_{k}  &  =\left\lceil \frac{1}{n-2k-2}\left(  \binom{n}{2}+(n-2k-3)u_{k-1}%
\right)  \right\rceil \text{ for }k\geq m\text{.}%
\end{align*}
The following is the new lower bound on $E_{\leq k}(  n)  $. It
follows from Theorem \ref{main}.

\begin{theorem}
\label{recursive}For any $n$ and $k$ such that $m-1\leq
k\leq(n-3)/2$,
\[
E_{\leq k}(n)\geq u_{k}.
\]
\end{theorem}

\begin{proof} We need the following two lemmas to estimate the growth
of the sequence $u_k$ with respect to $n$ and $k$. For presentation
purposes, we defer their proofs to the end of the section.

\begin{lemma}
\label{double ineq}For any $k$ such that $m-1\leq k\leq(n-5)/2,$
\begin{equation}
3\sqrt{1-\frac{2k+9/2}{n}}<\frac{\binom{n}{2}-u_{k}}{\binom{n}{2}-u_{m-1}}
\leq3\sqrt{1-\frac{2k+2}{n}}\text{.} \label{boundingtheus}
\end{equation}
\end{lemma}

\begin{lemma}
\label{estim}For any $k$ such that $m\leq k\leq(n-5)/2$,
\[
3\sqrt{1-\frac{2k+9/2}{n}}\left(  \binom{n}{2}-u_{m-1}\right)  \geq\left(
n-1\right)  \left(  n-2k-3\right)  .
\]
\end{lemma}

We prove the stronger statement $\widetilde{E}_{\leq k}( n) \geq
u_{k}$. Let $\Pi$ be an allowable sequence on $[ n] $ and $\pi$ any
of its halfperiods. We proceed by induction on $k$. If $k=m-1$ the
result holds by Inequality (\ref{lower}), proved in the more general
context of allowable sequences \cite{ABFLS2}. Assume that $k\geq m$
and $E_{\leq k-1}(\Pi)\geq u_{k-1}$. Let $s=s( k+1,\pi) $; by
Theorem \ref{main},
\[
E_{\geq k+1}\left(  \Pi\right)  \leq\left(  n-2k-3\right)  E_{k}\left(
\Pi\right)  -\frac{s}{2}\left(  E_{k}\left(  \Pi\right)  -\left(  n-1\right)
\right)  .
\]
If $s=0$ or $E_{k}(\Pi)\geq n-1,$ then $E_{\geq k+1}(\Pi
)\leq(n-2k-3)E_{k}(\Pi)$. Thus
\[
\binom{n}{2}-E_{\leq k}(\Pi)\leq\left(  n-2k-3\right)  \left(
E_{\leq k} (\Pi)-E_{\leq k-1}(\Pi)\right)  ,
\]
and by induction
\begin{align*}
E_{\leq k}(\Pi) &\geq\frac{1}{n-2k-2}\left( \binom{n}{2}+(n-2k-3)E_{\leq k-1}\left( \Pi\right)  \right)\\
&\geq\frac{1}{n-2k-2}\left( \binom{n}{2}+(n-2k-3)u_{k-1} \right) ,
\end{align*}
which implies that $E_{\leq k}(\Pi)\geq u_{k}$ by definition of
$u_{k}$. Now assume $s>0$ and $E_{k}(\Pi)<n-1$. Because $E_{k}( \Pi)
\geq2k+3$ (see the proof of Corollary \ref{coro: maxs}), it follows
that $k\leq(n-5)/2$. By Theorem
\ref{main},%
\begin{align*}
E_{\geq k+1}(\Pi)  & \leq(n-2k-3)E_{k}(\Pi)-\frac{s}{2}\left(  E_{k}%
(\Pi)-\left(  n-1\right)  \right)  \\
& =(n-2k-3-\frac{s}{2})E_{k}(\Pi)+\frac{s}{2}\left(  n-1\right).
\end{align*}
Recall that $s=s(k+1,\pi)\leq n-2k-3$. Because $E_{k}(\Pi)<n-1$, it
follows that
\begin{align*}
E_{\geq k+1}(\Pi) &  \leq(n-2k-3-\frac{s}{2})(n-1)+\frac{s}{2}\left(
n-1\right)  \\
&  =\left(  n-1\right)  \left(  n-2k-3\right)  \text{.}%
\end{align*}
Therefore%
\begin{equation*}
E_{\leq k}(\Pi)=\binom{n}{2}-E_{\geq k+1}(\Pi)\geq\binom{n}{2}-(n-1)\left(
n-2k-3\right)  \text{.}%
\end{equation*}
By Lemma \ref{estim},%
\[
E_{\leq k}(\Pi)\geq\binom{n}{2}-3\sqrt{1-\frac{2k+9/2}{n}}\left(  \binom{n}%
{2}-u_{m-1}\right)  ,
\]
and by Lemma \ref{double ineq}, $E_{\leq k}(\Pi)\geq u_{k}$ for all
allowable sequences $\Pi$ on $[  n]  $. Therefore $E_{\leq k}(n)\geq
\widetilde{E}_{\leq k}(n)\geq u_{k}$.
\end{proof}

\begin{corollary}
\label{explicit}For any $n$ and $k$ such that $m-1\leq
k\leq(n-2)/2$,
\[
E_{\leq
k}(n)\geq\binom{n}{2}-\frac{1}{9}\sqrt{1-\frac{2k+2}{n}}\left(
5n^{2}+19n-31\right).%
\]

\end{corollary}

\begin{proof}
Let $\Pi$ be an allowable sequence on $[n]$. If $k=\lfloor
n/2\rfloor -1$, then $E_{\leq \lfloor n/2\rfloor -1}( \Pi )
=\binom{n}{2}$. For $k<\lfloor n/2\rfloor -1$, it follows that
$n\geq3$ and from Theorem \ref{recursive} and Lemma \ref{double
ineq},
\[
E_{\leq k}(\Pi)\geq u_{k}\geq\binom{n}{2}-3\sqrt{1-\frac{2k+2}{n}}\left(
\binom{n}{2}-u_{m-1}\right)  .%
\]
Considering the possible residues of $n$ modulo $9$, it can be
verified that for $n\geq3$,
\[
u_{m-1} \geq \frac{17}{54} n^2-\frac{65}{54} n+\frac{31}{27}\text {
(equality if $n\equiv 3$ (mod 9))}.
\]
Therefore $E_{\leq k}(n)\geq\widetilde{E}_{\leq
k}(n)\geq\binom{n}{2}-\frac{1} {9}\sqrt{1-\frac{2k+2}{n}}(
5n^{2}+19n-31)  $.
\end{proof}

\subsection*{Proofs of the Lemmas}
\begin{proof}[Proof of Lemma \ref{double ineq}]
The integer range $[ m-1,(n-5)/2] $ is empty for $n\leq5$. Assume
$n\geq6$ and proceed by induction on $k$. If $k=m-1$, then
$3\sqrt{1-(2m+5/2)/n}\leq1\leq3\sqrt{1-2m/n}$ is equivalent to
$\lceil (  4n-11)  /9\rceil \leq4n/9\leq\lceil ( 4n-11) /9\rceil
+5/4$ which holds in general. Assume that $k\geq m$ and that
(\ref{boundingtheus}) holds for $k-1$. From the definition of $u_k$
and the induction hypothesis,
\begin{align*}
\binom{n}{2}-u_{k}  & \leq\binom{n}{2}-\frac{1}{n-2k-2}\left(
\binom{n}{2}+(n-2k-3)u_{k-1}\right)  \\
& =\frac{n-2k-3}{n-2k-2}\left(  \binom{n}{2}-u_{k-1}\right)
\leq3\left( \binom{n}{2}-u_{m-1}\right)
\frac{n-2k-3}{n-2k-2}\sqrt{1-\frac{2k}{n}},
\end{align*}
and $(n-2k-3)\sqrt{1-2k/n}\;/ (n-2k-2)\leq \sqrt{1-(2k+2)/n}$
because $k \leq (n-5)/2$, which proves the second inequality in
(\ref{boundingtheus}). Similarly, from the definition of $u_k$ and
the induction hypothesis,
\begin{align*}
\binom{n}{2}-u_{k}  & \geq\binom{n}{2}-\frac{1}{n-2k-2}\left(
\binom{n}{2}+(n-2k-3)u_{k-1}\right)  -1\\
& =\frac{n-2k-3}{n-2k-2}\left(  \binom{n}{2}-u_{k-1}\right)
-1\geq3\left( \binom{n}{2}-u_{m-1}\right)
\frac{n-2k-3}{n-2k-2}\sqrt{1-\frac{2k+5/2}{n}}-1.
\end{align*}


Hence, to prove the second inequality in (\ref{boundingtheus}), it
is enough to show that $3(  \binom{n}{2}-u_{m-1})  d>1,$ where
\begin{equation}
d=\frac{n-2k-3}{n-2k-2}\sqrt{1-\frac{2k+5/2}{n}}-\sqrt{1-\frac{2k+9/2}{n}}
\label{definitiond}
\end{equation}
is always positive because $k\leq(n-5)/2$.
First note that%
\[
u_{m-1}\leq3\binom{m+1}{2}+3\binom{m+1-\lfloor n/3\rfloor}{2}\leq
3\binom{(4n+6)/9}{2}+3\binom{(n+10)/9}{2},
\]
which implies that%
\begin{equation}
3\left(  \binom{n}{2}-u_{m-1}\right)  \geq\frac{1}{9}\left(  5n^{2}%
-25n+4\right)  .\label{boundbinomial-u}%
\end{equation}
Multiplying the easily-verified inequality
\[
1>\frac{\left(  n-2k-3\right)  \sqrt{n-2k-5/2}+\left(  n-2k-2\right)
\sqrt{n-2k-9/2}}{\left(  2n-4k-5\right)  \sqrt{n-2k-5/2}}%
\]
by Identity (\ref{definitiond}), yields
\begin{align*}
d &  >\frac{n-2k-9/4}{\left(  n-2k-2\right)  ^{2}\sqrt{n\left(
n-2k-5/2\right)  }}\cdot\frac{2n-4k-4}{2n-4k-5}\label{d ineq}\\
&  >\frac{n-2k-9/4}{\left(  n-2k-2\right)  ^{2}\sqrt{n\left(
n-2k-5/2\right)
}}\nonumber\\
&  =\left(  1-\frac{1}{4\left(  n-2k-2\right)  }\right)
\frac{1}{\left( n-2k-2\right)  \sqrt{n\left(  n-2k-2-1/2\right)
}}.\nonumber
\end{align*}
Since $(4n-11)/9\leq k\leq(n-5)/2$, then $3\leq n-2k-2\leq(n+4)/9$. Thus%
\[
d>\left(  1-\frac{1}{12}\right)  \frac{27}{\left(  n+4\right)
\sqrt{n\left( n-1/2\right)  }}=\frac{99}{4\left(  n+4\right)
\sqrt{n\left(  n-1/2\right). }}%
\]
This inequality, together with Inequality (\ref{boundbinomial-u}), imply that for all $n\geq6$,%
\[
3\left(  \binom{n}{2}-u_{m-1}\right)  d>\frac{11}{4}\left(  \frac
{5n^{2}-25n+4}{\left(  n+4\right)  \sqrt{n\left(  n-1/2\right)
}}\right)
>1.\qedhere
\]
\end{proof}

\begin{proof}[Proof of Lemma \ref{estim}]
For each $n\leq40$ the integer range $[ m,(n-5)/2] $ is either empty
or contains only $k=\lfloor(n-5)/2\rfloor $. For these cases, the
inequality can easily be verified. Assume $n\geq41$, it follows from
Inequality (\ref{boundbinomial-u}) that
\[
9\left(  1-\frac{2k+9/2}{n}\right)  \left(
\binom{n}{2}-u_{m-1}\right) ^{2}\geq\frac{(n-2k-9/2)\left(
5n^{2}-25n+4\right)  ^{2}}{81n}.
\]
Since $k\leq(n-5)/2$, then
\[
n-2k-9/2\geq\frac{(n-2k-3)^{2}}{n-2k+3}.
\]
Also $k\geq m\geq(4n-11)/9$ implies $n-2k+3\leq(n+49)/9$ and thus
\[
\frac{(n-2k-9/2)\left(  5n^{2}-25n+4\right)  ^{2}}{81n}\geq\frac
{(n-2k-3)^{2}\left(  5n^{2}-25n+4\right)  ^{2}}{9n\left(
n+49\right)
}\text{.}%
\]
Finally, for $n\geq41$,
\[
\frac{\left(  5n^{2}-25n+4\right)  ^{2}}{9n\left(  n+49\right)
}\geq (n-1)^{2},
\]
and consequently%
\[
9\left(  1-\frac{2k+9/2}{n}\right)  \left(
\binom{n}{2}-u_{m-1}\right) ^{2}\geq(n-1)^{2}(n-2k-3)^{2}.\qedhere
\]
\end{proof}

\section{\label{lower crossings}New lower bound on $\rcr(n)$}

In this section, we use Corollary \ref{explicit} to get the
following new lower bound on $\rcr(n)$.

\begin{theorem}
$\rcr(n)\geq\frac{277}{729}\binom{n}{4}+\Theta(n^{3})>0.379972\binom
{n}{4}+\Theta(n^{3})$. \label{th: crossing}
\end{theorem}

\begin{proof}
We actually prove that the right hand side is a lower bound on
$\pscr(n)$. According to (\ref{crossingsvsksets}), if $\Pi$ is an
awollable sequence on $[  n]  $, then
\[
\crg\left(  \Pi\right)  =\binom{n}{4}\left(  24
{\displaystyle\sum\limits_{k=0}^{\left\lfloor n/2\right\rfloor -1}}
\frac{1}{n}\left(  1-\frac{2k}{n}\right)  \frac{E_{\leq k}\left(
\Pi\right) }{n^{2}}\right)  +\Theta\left(  n^{3}\right).
\]
Using Inequality (\ref{lower}) for $0\leq k\leq m-1$ gives
\[
\frac{E_{\leq k}\left(  \Pi\right)  }{n^{2}}\geq\frac{3}{2}\left(  \frac{k}%
{n}\right)  ^{2}+\frac{3}{2}\max\left(
0,\frac{k}{n}-\frac{1}{3}\right) ^{2}-\Theta\left(
\frac{1}{n}\right).
\]
Similarly, if $m\leq k\leq\lfloor n/2\rfloor -1$, then by Corollary
\ref{explicit},%
\[
\frac{E_{\leq k}\left(  \Pi\right)  }{n^{2}}\geq\frac{1}{2}-\frac{5}{9}%
\sqrt{1-\frac{2k}{n}}+\Theta\left(  \frac{1}{n}\right)  \text{.}%
\]
Therefore,%
\begin{align*}
\crg\left(  \Pi\right)   &  \geq\binom{n}{4}\left(  24\int_{0}^{4/9}\frac{3}%
{2}(1-2x)\left(  x^{2}+\max\left(  0,x-\frac{1}{3}\right)  ^{2}\right)
dx\right)  \\
&  +\binom{n}{4}\left(  24\int_{4/9}^{1/2}\left(  1-2x\right)  \left(
\frac{1}{2}-\frac{5}{9}\sqrt{1-2x}\right)  dx\right)  +\Theta(n^{3})\\
&  \geq\binom{n}{4}\left(  \frac{86}{243}+\frac{19}{729}\right)
+\Theta
(n^{3})=\frac{277}{729}\binom{n}{4}+\Theta(n^{3})\text{.}\qedhere
\end{align*}
\end{proof}

The following is the list of best lower bounds for $\pscr( n) $ in
the range $28\leq n\leq99$ that follow from using Identity
(\ref{crossingsvsksets}) with the bound in either Inequality
(\ref{lower}) or the new bound from Theorem \ref{recursive}.

\begin{center}%
\begin{tabular}
[c]{l|l||l|l||l|l||l|l||l|l||l|l}%
$n$ & $\pscr\left(  n\right)  \geq$ & $n$ & $\pscr\left( n\right)
\geq$ & $n$ & $\pscr\left(  n\right)  \geq$ & $n$ & $\pscr\left(
n\right)  \geq$ & $n$ & $\pscr\left( n\right)  \geq$ & $n$ &
$\pscr\left( n\right) \geq$\\\hline\hline $28$ & 7233 & $40$ & 33048
& $52$ & 99073 & $64$ & 234223 & $76$ & 475305 &
$88$ & 866947\\
$29$ & 8421 & $41$ & 36674 & $53$ & 107251 & $65$ & 249732 & $77$ & 501531 &
$89$ & 907990\\
$30$ & 9723 & $42$ & 40561 & $54$ & 115878 & $66$ & 265888 & $78$ & 528738 &
$90$ & 950372\\
$31$ & 11207 & $43$ & 44796 & $55$ & 125087 & $67$ & 282974 & $79$ & 557191 &
$91$ & 994394\\
$32$ & 12830 & $44$ & 49324 & $56$ & 134798 & $68$ & 300767 & $80$ & 586684 &
$92$ & 1039840\\
$33$ & 14626 & $45$ & 54181 & $57$ & 145030 & $69$ & 319389 & $81$ & 617310 &
$93$ & 1086725\\
$34$ & 16613 & $46$ & 59410 & $58$ & 155900 & $70$ & 338913 & $82$ & 649190 &
$94$ & 1135377\\
$35$ & 18796 & $47$ & 65015 & $59$ & 167344 & $71$ & 359311 & $83$ & 682308 &
$95$ & 1185551\\
$36$ & 21164 & $48$ & 70948 & $60$ & 179354 & $72$ & 380531 & $84$ & 716507 &
$96$ & 1237263\\
$37$ & 23785 & $49$ & 77362 & $61$ & 192095 & $73$ & 402798 & $85$ & 752217 &
$97$ & 1290844\\
$38$ & 26621 & $50$ & 84146 & $62$ & 205437 & $74$ & 425980 & $86$ & 789077 &
$98$ & 1346029\\
$39$ & 29691 & $51$ & 91374 & $63$ & 219457 & $75$ & 450078 & $87$ & 827289 &
$99$ & 1402932
\end{tabular}

\end{center}

\section{A point-set with few $(\le k)$-edges for every $k \le
{4n/9}-1$}\label{constructions}

Combining Inequality (\ref{lower}) and Theorem~\ref{recursive}, we
obtain the best known lower bound for $E_{\le k}(n)$. If $n$ is a
multiple of $9$ and $k \le (4n/9)-1$, then this bound reads

\begin{align}
{E_{\leq k}(n)}\ge %
\begin{cases}
3\binom{k+2}{2} & \text{\quad if $0 \le k\leq n/3-1$,} \\[0.2cm]
3\binom{k+2}{2}+3\binom{k-n/3+2}{2} & \text{\quad if $n/3 \le k\leq
  4n/9-2$,} \\[0.2cm]
3\binom{(4n/9-1)+2}{2}+3\binom{(4n/9-1)-n/3+2}{2}+3
& \text{\quad if $k=4n/9-1$.}
\end{cases}%
\label{eq:thelowerbound}
\end{align}

Our aim in this section is to show that this bound is tight for $n
\ge 27$.  This improves on the construction in~\cite{AGOR3}, where
tightness for Inequality (\ref{eq:thelowerbound}) is proved for $k
\le (5n/12)$.

We recursively construct, for each integer
$\eme\ge 3$, a $9\eme$-point set $S_{\eme}$ such that for every $k
\le (4n/9)-1$, $E_{\le k}(S_{\eme})$
equals the right hand side of~(\ref{eq:thelowerbound}).



\subsubsection*{Constructing the sets $S_\eme$}

If $a$ and $b$ are distinct points, then $\edge{ab}$ denotes the
line spanned by $a$ and $b$, and $\oL{ab}$ denotes the closed line
segment with endpoints $a$ and $b$,  directed from $a$ towards $b$.
Let $\theta$ denote the clockwise rotation by an angle of $2\pi/3$
around the origin. At this  point  the reader may want to take a
sneak preview at Figure~\ref{fig:figure2}, where $S_3$ is sketched.

For each $\eme \ge 3$ the set $S_{\eme}$ is naturally partitioned
into nine sets of size $r$: $A_\eme=\{a_1,\ldots,a_\eme\}$,
$A_\eme'=\{a_1',\ldots,a_\eme'\}$, $A_\eme''$, and their respective
$2\pi/3$ and $4\pi/3$ rotations around the origin. The elements of
$A_\eme''$ are not labeled because they change in each iteration.
For $i=1,\ldots,r$, we let $b_i=\theta(a_i), b_i'=\theta(a_i'),
c_i=\theta^2(a_i)$, and $c_i'=\theta^2(a_i)$. Thus if we let
$B_\eme=\{b_1,\ldots,b_\eme\}$, $B_\eme'=\{b_1',\ldots,b_\eme'\}$,
$B_\eme''=\theta(A_\eme'')$, $C_\eme=\{c_1,\ldots,c_\eme\}$,
$C_\eme'=\{c_1',\ldots,c_\eme'\}$, and
$C_\eme''=\theta^2(A''_\eme)$, then we obtain $B_\eme\cup
B_\eme'\cup B_\eme''$ (respectively, $C_\eme\cup C_\eme'\cup
C_\eme''$) by applying $\theta$ (respectively, $\theta^2$) to
$A_\eme\cup A_\eme'\cup A_\eme''$. We refer to this property as the
$3$-{\em symmetry} of $S_{\eme}$.

As we mentioned before, the construction of the sets $S_\eme$ is
recursive. For $\eme \ge 3$, we obtain $A_{\eme+1}$ and
$A_{\eme+1}'$ by adding suitable points $a_{\eme+1}$ to $A_{\eme}$
and $a'_{\eme+1}$ to $A_{\eme}'$. Keeping $3$-symmetry, this
determines $B_{\eme+1}$, $B'_{\eme+1}$, $C_{\eme+1}$, and
$C'_{\eme+1}$. However, the set $A''_{\eme+1}$ is {\em not} obtained
by adding a point to $A''_\eme$, but instead is defined in terms of
$B_{\eme+1},B'_{\eme+1}, C_{\eme+1}$, and $C'_{\eme+1}$; this
explains why we have not listed the elements in $A''_\eme,
B''_\eme$, and $C''_\eme$.

Before moving on with the construction, we remark that the sets
$S_\eme$ contain subsets of more than two collinear points. As it
will become clear from the construction, the points can be slightly
perturbed to general position, so that the number of $(\le k)$-edges
remains unchanged for every $k \le 4n/9-1$.

\begin{figure}[hp]
\begin{center}
\includegraphics[scale=.95]{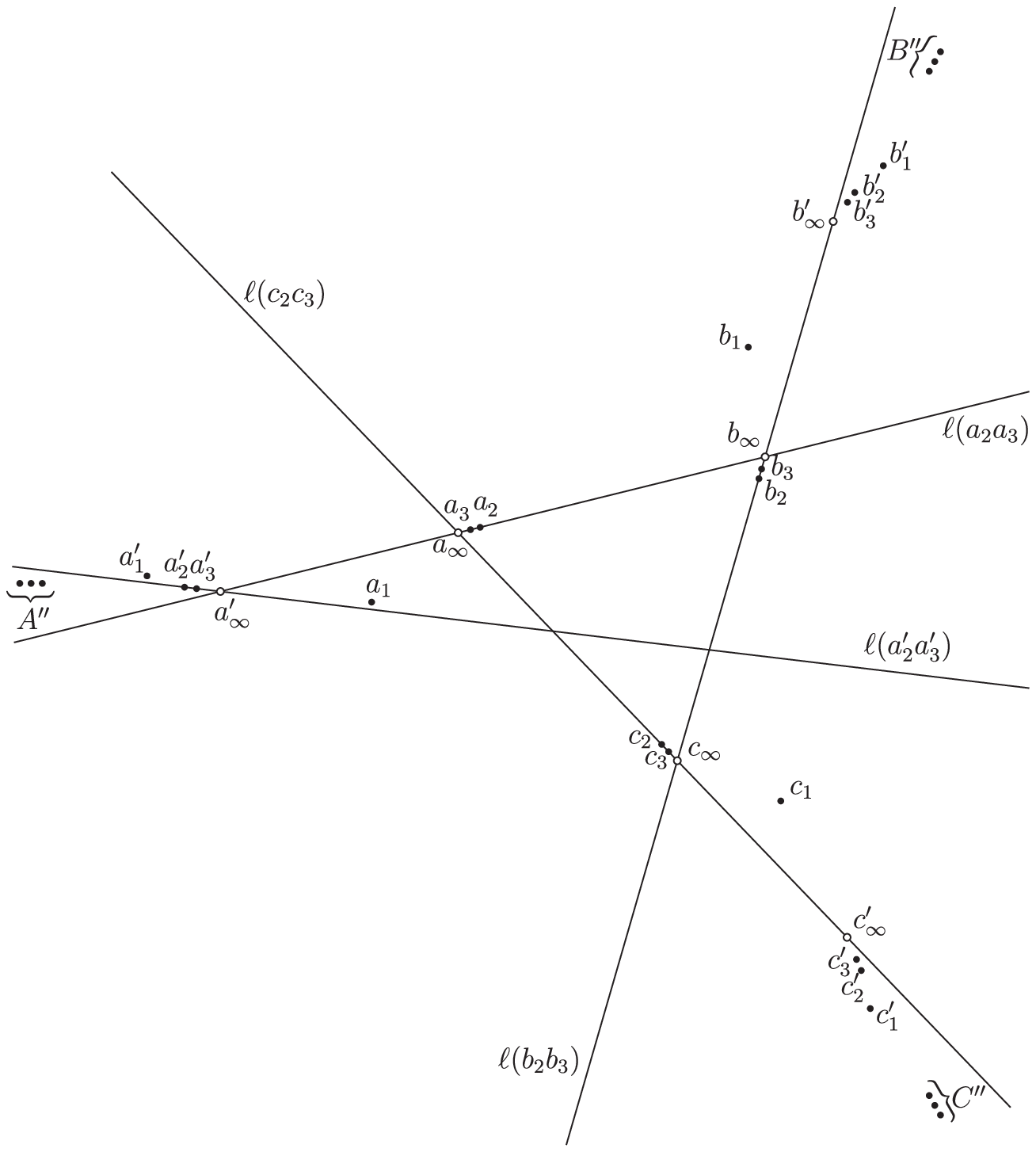}
\caption{The $27$-point set $S_{3}$. The points $a_\infty,
a'_\infty,  b_\infty, b'_\infty, c_\infty,$ and $c'_\infty$ do not
belong to  $S_3$.} \label{fig:figure2}
\end{center}
\end{figure}

We start by describing $S_{3}$, see Figure~\ref{fig:figure2}. First
we explicitly fix $A_3$ and $A_3'$: $a_1 = (-700, -50)$, $a_2 =
(-410 ,150)$, $a_3 = (-436 , 144 )$, $a'_1=( -1300, 20 )$, $a'_2
=(-1200,-10 )$, and $a'_3=(-1170 ,-14 )$. Thus $B_3, B_3',C_3,$ and
$C_3'$ also get determined. For the points in $A_3''$ we do not give
their exact coordinates, instead we simply ask that they satisfy the
following: all the points in $A_3''$ lie on the $x$-axis, and are
sufficiently far to the left of $A_{3} \cup A_{3}'$ so that if a
line $\ell_1$ passes through a point in $A_3''$ and a point in
$S_3\setminus {(B_3''\cup
  C_3'')} $, and a line $\ell_2$ passes through two points in $S_3\setminus{A_3''}$, then
the slope of $\ell_1$ is smaller in absolute value than the slope of
$\ell_2$, i.e., $\ell_1$ is closer (in slope) to a horizontal line,
than $\ell_2$.

\begin{figure}[h]
\begin{center}
\includegraphics[scale=1]{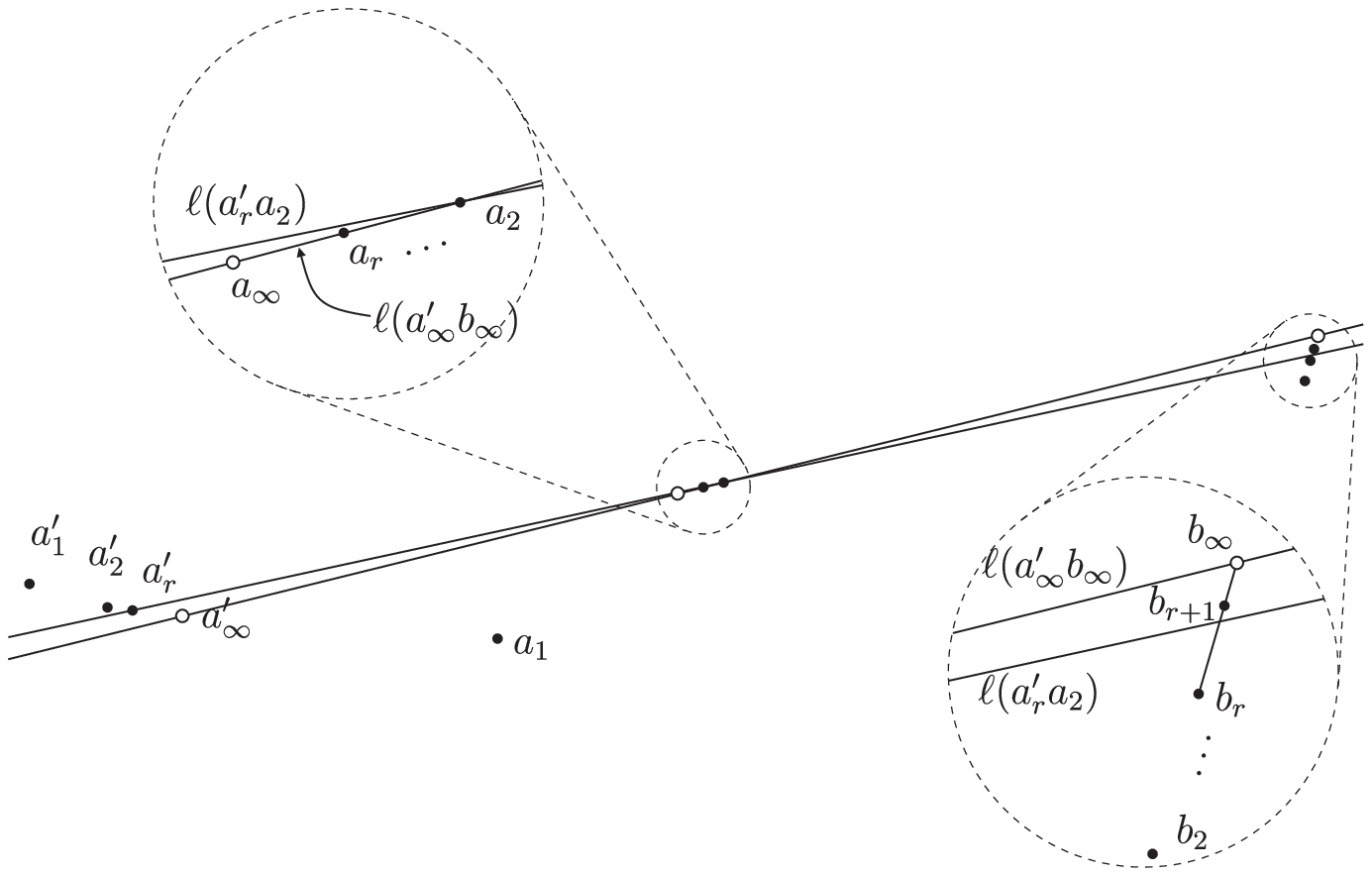}
\caption{$b_{\eme+1}$ is placed in between $b_\eme$
  and $b_\infty$, above the line $\edge{a'_\eme a_2}$.}
\label{fig:figure3}
\end{center}
\end{figure}

We need to define six auxiliary points
not in $S_\eme: a_\infty = \edge{a_2 a_3} \cap \edge{c_2 c_3}$ and
$a'_\infty = \edge{a'_2 a'_3} \cap \edge{a_2 a_3}$. As
expected, let $b_\infty=\theta(a_\infty),
c_\infty=\theta^2(a_\infty), b'_\infty=\theta(a'_\infty),$ and
$c'_\infty=\theta^2(a'_\infty)$.

We now describe how to get $S_{\eme+1}$ from $S_{\eme}$. The crucial
step is to define the points $b_{\eme+1}$ and $a'_{\eme+1}$ to be
added to $B_\eme$ and $A'_{\eme}$ to obtain $B_{\eme+1}$ and
$A'_{\eme+1}$, respectively. Then we construct $A''_{\eme+1}$ and
applying $\theta$ and $\theta^2$ to $B_{\eme+1}$, $A'_{\eme+1}$, and
$A''_{\eme+1}$, we obtain the rest of $S_{\eme+1}$.

Suppose that for some $\eme \ge 3$, the set $S_\eme$ has been
constructed so
that the following properties hold for
$\te = \eme$
(this is clearly true for the base case $\eme=3$):

\begin{description}
\item{(I)} The points $a_2,\ldots,a_\te$ appear in
this order along $\oL{a_2 a_\infty}$.
\item{(II)} The points
$a_2',\ldots,a_\te'$ appear in this order along $\oL{a'_2
  a'_\infty}$.
\item{(III)} For all $i=2,\ldots,\te-1$ and $j=2,\ldots,\te$,  $\edge{a'_i a_j}$
  intersects the interior of  $\oL{b_i b_{i+1}}$.
\item{(IV)} For all $j=2,\ldots,\te$,  $\edge{a'_\te a_j}$
  intersects the interior of  $\oL{b_\te b_{\infty}}$.
\end{description}

Now we add $b_{\eme+1}$ and $a'_{\eme+1}$. Place $b_{\eme+1}$
anywhere on the open line segment determined by $b_\infty$ and the
intersection point of $\edge{a_\eme'a_2}$ with $\oL{b_\eme
b_\infty}$. (The existence of this intersection point is guaranteed
by (IV), see Figure~\ref{fig:figure3}). Place $a'_{\eme+1} $
anywhere on the open line segment determined by $a'_\infty$ and the
intersection point of $\edge{b_{\eme+1} a_\infty}$ with $\oL{a'_\eme
a'_\infty}$. (This intersection exists because $a_\infty', a_\infty,
a_2$, and $b_\infty$ are collinear and appear in this order along
$\edge{a_\infty' b_\infty}$, the line $\edge{a_\infty' b_\infty}$
separates $b_{r+1}$ from $a_\eme'$, and the line $\edge{a_\eme'
a_2}$ separates $b_{r+1}$ from $a_\infty$, see
Figure~\ref{fig:figure4}). Thus $B_{\eme+1}$ and $A'_{\eme+1}$ and
consequently $A_{\eme+1}, C_{\eme+1}, B'_{\eme+1}$, and
$C'_{\eme+1}$, are defined. It is straightforward to check that
(I)--(IV) hold for $t = \eme+1$.

\begin{figure}[htbp]
\begin{center}
\includegraphics[scale=1]{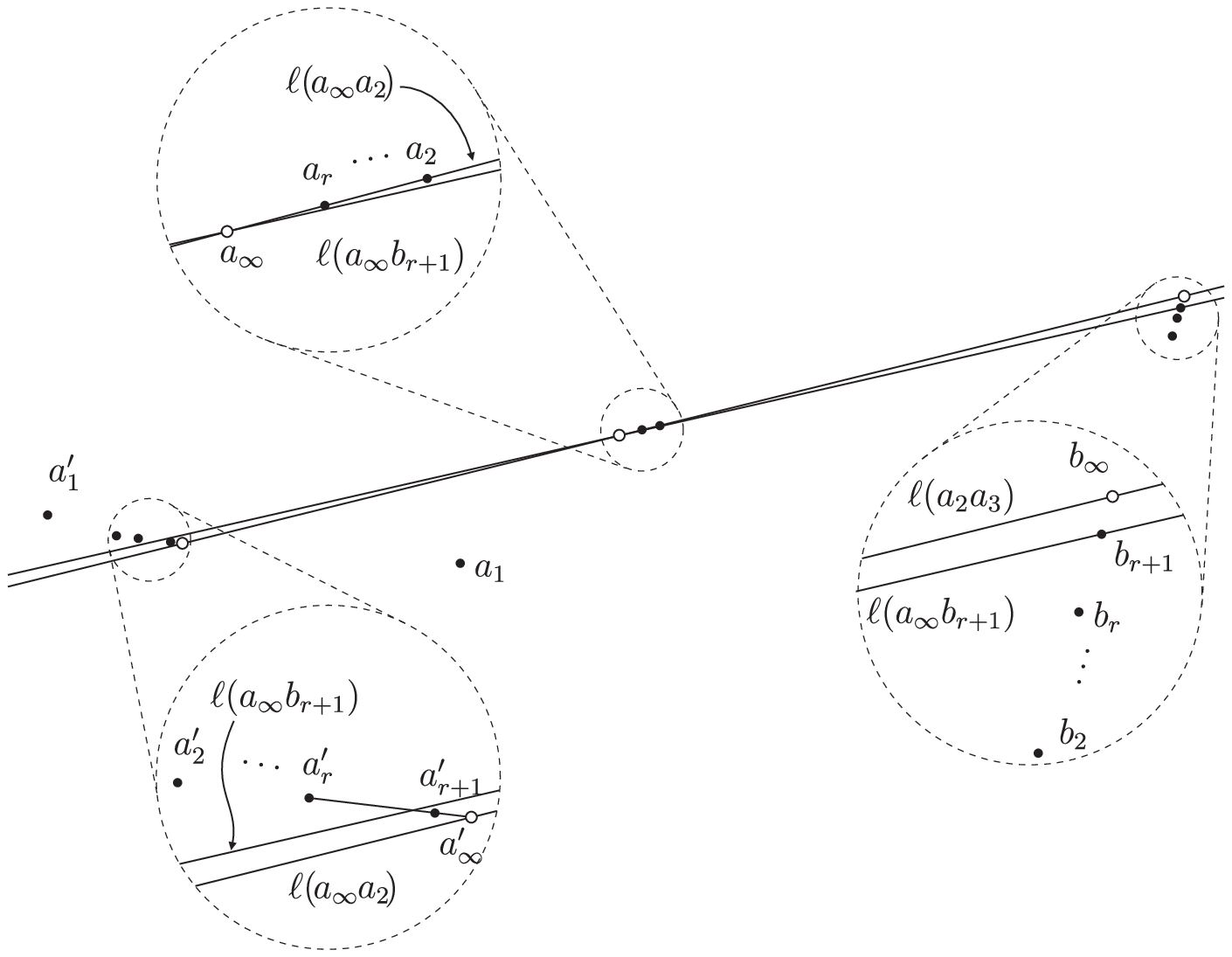}
\caption{$a'_{\eme+1}$ is placed in between $a'_\eme$
  and $a'_\infty$, below the line $\edge{a_\infty b_{\eme+1}}$.}
\label{fig:figure4}
\end{center}
\end{figure}

It only remains to describe how to construct $A''_{\eme+1}$. As we
mentioned above, this set is not a superset of $A''_{\eme}$,
instead it gets defined analogously to $A''_3$: we let the points
in $A_{\eme+1}''$ lie on the $x$-axis, and sufficiently far to the left of $A_{\eme+1} \cup
A_{\eme+1}'$, so that if $\ell_1$ passes through a point in $A_{\eme+1}''$ and
through a point in $S_{\eme+1}\setminus {(B_{\eme+1}''\cup C_{\eme+1}'')} $,
and $\ell_2$ spans two points in $S_{\eme+1}\setminus{A_{\eme+1}''}$,
then the slope of $\ell_1$ is smaller in absolute value than the slope
of $\ell_2$.

\subsubsection*{ Calculating $E_{\le k}(S_\eme)$}

We fix $\eme\ge 3$, and proceed to determine $E_{\le k}(S_{\eme})$
for each $k$, $0\le k \le 4\eme-1$. It is now convenient to label
the elements of $A''_{\eme}, B''_{\eme}$, and $C''_{\eme}$. Let
$a''_1,a''_2,\ldots,a''_r$ be the elements of $A''_{\eme}$, ordered
as they appear from left to right along the negative $x$-axis. As
expected, let $b''_i=\theta(a''_i)$ and $c''_i=\theta^2(a''_i)$, for
$i=1,\ldots,\eme$.

We call a $k$-edge {\em bichromatic} if
it joins two points with different label letters (i.e.,
if it is of the form $ab, bc$, or $ac$); otherwise, a $k$-edge
is {\em monochromatic}.
A monochromatic edge is {\em of type $aa$} if
it is of the form $\edge{a_i a_j}$ for some integers $i,j$; edges of
types $aa', aa'', a'a', a'a'', a''a''$ (and their counterparts for $b$
and $c$) are similarly defined. Finally, we say that an edge of any of
the types $aa, aa', aa'', a'a', a'a''$, or $a''a''$ is {\em of type}
$\mathbf{A}$; edges of types $\mathbf{B}$ and $\mathbf{C}$ are
similarly defined. We let $\ebic{k}$ (respectively,
$\emono{k}$) stand for the number of bichromatic (respectively,
monochromatic) $(\le k)$-edges, so that $E_{\le k}(S_{\eme})
= \ebic{k}(S_{\eme}) +  \emono{k}(S_{\eme})$.

 We say that a finite point set $P$ is $3$\emph{-decomposable} if
 it can be partitioned into three equal-size sets $\overline{A}$,
 $\overline{B}$, and $\overline{C}$ satisfying
the following: there is a triangle $T$ enclosing $P$ such that the
orthogonal projections of $P$ onto the three sides of $T$ show $\overline{A}$
between $\oL{B}$ and $\oL{C}$ on one side, $\oL{B}$ between $\oL{A}$ and $\oL{C}$ on another
side, and $\oL{C}$ between $\oL{A}$ and $\oL{B}$ on the third side
(see~\cite{ACFLS}). We say that $\{\oL{A},\oL{B},\oL{C}\}$ is a
$3$-{\em decomposition} of $P$. It is easy to see that if we let
$\oL{A}:=A_\eme \cup A_\eme'\cup A_\eme''$,
$\oL{B}:=B_\eme \cup B_\eme'\cup B_\eme''$, and
$\oL{C}:=C_\eme\cup C_\eme'\cup C_\eme''$,
then $\{\oL{A}, \oL{B}, \oL{C}\}$ is a $3$-decomposition of $S_{\eme}$:
indeed, it suffices to take an enclosing triangle of $S_\eme$ with
one side orthogonal to the line spanned by the points in $A''$,
one side orthogonal to the line spanned by the points in $B''$, and
one side orthogonal to the line spanned by the points in $C''$.
Thus, it follows from Claim 1
  in~\cite{ACFLS} (where it is proved in the more general setting of
  allowable sequences) that
\begin{align}
{\ebic{k}(S_{\eme})}=%
\begin{cases}
3\binom{k+2}{2},&  \text{\quad if $0 \le k\leq 3\eme-1$;} \\[0.2cm]
3\binom{3\eme+1}{2}+(k-3\eme+1)9r,&  \text{\quad if $3\eme\le k\le 4\eme -1$.}%
\end{cases}%
\label{eq:thebic}
\end{align}

We now count the monochromatic $(\le k)$-edges. By $3$-symmetry, it suffices to
focus on those {of type} $\mathbf{A}$.


It is readily checked that for all $i$ and $j$ distinct integers,
$\edge{a_i
  a_j}, \edge{a'_i a'_j}$, and $\edge{a''_i a''_j}$ are $k$-critical edges for
some $k > 4\eme-1$. The same is true for $\edge{a_i a'_j}$ whenever
$i$ and $j$ are not both equal to $1$ (when $i\neq 1$ and $j\neq 1$
this follows from (III) and (IV) ), while $\edge{a_1
  a'_1}$ is a $(4\eme-1)$-edge. Now, for each $i,j$, $1 \le i \le
\eme$, $2\le j \le \eme$, $\edge{a_i'' a'_j}$ is a
$(4\eme+i-j)$-edge,
while $a_i'' a'_1$ is a $(4\eme+i-2)$-edge. Finally, if $1\le i \le
\eme$ and $2\le j \le \eme$, then $\edge{a_i'' a_j}$ is a
$(3\eme+i+j-3)$-edge, and $\edge{a_i'' a_1}$ is a
$(3\eme+i-1)$-edge.  In conclusion (to obtain (i),
we recall that a $k$-edge is also a $(9\eme-2-k)$-edge):
\begin{description}
\item{(i)} for $1 \le s \le \eme$, the
number of  $(3\eme-1+\ese)$-edges of types $a'a''$ or
$aa''$ is $2s$;
\item{(ii)} there is exactly one $(4\eme-1)$-edge of type $aa'$;
and
\item{(iii)} all other edges of type $\mathbf{A}$ are $k$-critical edges for some
$k > 4\eme -1$.
\end{description}

It follows that the number of $(\le k)$-edges of type $A$ is
\begin{description}
\item{(a)} $0$, for
$k \le 3\eme-1$;
\item{(b)} $2\sum_{s=1}^{k-(3r-1)} s=2 {\tbinom{k-3r+2}{2}} $,
for $3\eme \le k \le 4\eme-2$;
\item{(c)} $1 + 2\sum_{s=1}^{(4r-1)-(3r-1)} s = 2 {\tbinom{r+1}{2}} + 1$, for $k = 4\eme-1$.
\end{description}

By $3$-symmetry, for each integer $k$ there are exactly as many $(\le
k)$-edges of type $\mathbf{A}$ as there are of type $\mathbf{B}$, and of
type $\mathbf{C}$. Therefore

\begin{align}
{\emono{k}(S_{\eme})}=%
\begin{cases}
0 &\text{\quad if $0 \le k\leq 3\eme-1$,} \\[0.2cm]
6 {k-(3\eme-2)\choose 2} &\text{\quad if $3\eme \le k\leq 4\eme-2$,} \\[0.2cm]
6 {r+1 \choose 2} + 3 &\text{\quad if $k=4\eme-1$.}
\end{cases}%
\label{eq:themono}
\end{align}

Because $E_{\le k}(S_{\eme}) = \ebic{k}(S_{\eme}) +
\emono{k}(S_{\eme})$, it follows by identities ~(\ref{eq:thebic})
and~(\ref{eq:themono}) that $E_{\le k}(S_{\eme})$ equals the right
hand side of~(\ref{eq:thelowerbound}).

\section{Concluding remarks}

The Inequality in Theorem \ref{main} is best possible. That is,
there are $n$-point sets $P$ whose simple allowable sequence $\Pi$
gives equality in the Inequality of Corollary \ref{coro: maxs}:
\[
E_{\geq k}( \Pi ) =  ( n-2k-1) E_{k-1}( \Pi ) + \binom{s}{2}.
\]
We present two constructions. The first has $s=n-2k-1$ and consists
of $2k+1$ points which are the vertices of a regular polygon and
$n-2k-1$ central points very close to the center of the polygon.
This construction was given in \cite{LVWW} to show that $E_{k-1}
\geq 2k+1$ is best possible. Indeed, note that the $(k-1)$-edges of
$P$ correspond to the larger diagonals of the polygon, and so
$E_{k-1}( \Pi )=2k+1$; moreover, any edge formed by two points in
the central part or one point in the central part and a vertex of
the polygon determine a $(\geq k)$-edge. Thus $E_{\geq k}( \Pi )
=\tbinom{n-2k-1}{2}+(2k+1)(n-2k-1)$, which achieves the desired
equality.

The second construction has $s=0$ and thus it can only be achieved
when $k \geq n/3$. Consider a $(2t+1)$-regular polygon where each
vertex is replaced by a set of $m$ points on a small segment
pointing in the direction of the center of the polygon. Let $\Pi$ be
the allowable sequence corresponding to this point-set, $n=(2t+1)m$,
and $k=tm$. It is straightforward to verify that
$E_{k-1}(\Pi)=(2t+1)m$ and $E_{\geq k} (\Pi)=2(2t+1)\tbinom{m}{2}$.
Thus $E_{\geq k}(\Pi)=(m-1) E_{k-1}(\Pi)=(n-2k-1)E_{k-1}(\Pi)$.

Prior to this work, there were two results that provided a lower
bound for $E_{\leq k}(P)$ based on the behavior of values of $k$
close to $n/2$. First, Welzl \cite{W} as a particular case of a more
general result proved that $E_{\leq k}(P) \geq F_1(k,n)$, where
\[
F_1(k,n)=\binom{n}{2}-2n\left( \sum_{j=k+1}^{n/2}k\right)
^{1/2}<\binom{n}{2}- \frac{\sqrt{2}}{2}n^{3/2}\sqrt{n-2k}.
\]
Second, Balogh and Salazar \cite{BS} proved that $E_{\leq k}(P) \geq
F_2(k,n)$, where $F_2(k,n)$ is a function that, for $n/3 \leq k \leq
n/2$, satisfies that
\begin{equation*}
F_2(k,n) < \binom{n}{2}- \frac{13\sqrt{3}}{36}n^{3/2}
\sqrt{n-2k}+o(n^{2}).
\end{equation*}%
By direct comparison, it follows that both $F_1(k,n)$ and $F_2(k,n)$
are smaller than the bound in Corollary \ref{explicit}. Thus our
bound is better than these two previous bounds.

 A nice feature of Theorem \ref{main} is that
it can give better bounds for $E_{\leq k}(n)$ and $k$ large enough,
and for $\rcr(n)$, provided someone finds a better bound than
Inequality (\ref{lower}) for $E_{\leq k}(n)$ when $4n/9<k<n/2$. For
example, \'Abrego et al. \cite{AFLS} considered $3$-regular point
sets $P$. These are point-sets with the property that for $1\leq j
\leq n/3$, the $j$th depth layer of $P$ has exactly 3 points of $P$.
A point $p\in P$ is in the $j$th depth layer if $p$ belongs to a
$(j-1)$-edge but not to a $(\leq j-2)$-edge of $P$. If $n$ is a
multiple of 18, they proved the following lower bound:
\begin{equation}
E_{\leq k}(P)  \geq 3\binom{k+2}{2}+3\binom{k+2- n/3}{2}+18
\binom{k+2- 4n/9}{2}.\label{eq:third binomial}
\end{equation}
This is better than the bound in Theorem \ref{recursive} for
$k>4n/9$, however using Theorem \ref{main} it is possible to find an
even better lower bound when $k\geq 17n/36$. We construct a new
recursive sequence $ u^\prime$ starting at $m=17n/36$ given by
\begin{align*}
u^\prime_{m-1}  & =3\binom{m+1}{2}+3\binom{m+1-\lfloor
n/3\rfloor}{2}+18 \binom{m+1-\lfloor 4n/9 \rfloor}{2}  \text{ and}\\
u^\prime_{k}  &  =\left\lceil \frac{1}{n-2k-2}\left(  \binom{n}{2}+(n-2k-3)u^\prime_{k-1}%
\right)  \right\rceil \text{ for }k\geq m\text{.}
\end{align*}
The value of $m=17n/36$ is the smallest possible for which
$u^\prime_m$ is greater than the right-side of Inequality
(\ref{eq:third binomial}). Following the proof of Theorem
\ref{recursive} it is possible to show that $E_{\leq k}(P) \geq
u^\prime_{k}$ for $17n/36 \leq k <n/2$. Thus, if we could show that
Inequality (\ref{eq:third binomial}) holds for arbitrary point sets
$P$, then we know that bound will no longer be tight for $k \geq
17n/36$. From equivalent statements to lemmas \ref{double ineq} and
\ref{estim}, it follows that $u^\prime_k \sim
\tbinom{n}{2}-(7\sqrt{2} n^2/18)\sqrt{1-2k/n}$. This in turn
improves the crossing number of $3$-regular point-sets $P$ to
$\rcr(P)\geq 0.380024\tbinom{n}{4}+\Theta(n^3)$.

In \cite{ACFLS} we considered other class of point-sets called
$3$-decomposable. These are point-sets $P$ for which there is a
triangle $T$ enclosing $P$ and a balanced partition $A$, $B$, and
$C$ of $P$, such that the orthogonal projections of $P$ onto the
sides of $T$ show $A$ between $B$ and $C$ on one side, $B$ between
$A$ and $C$ on another side, and $C$ between $A$ and $B$ on the
third side. For $3$-decomposable sets $P$ we were able to prove a
lower bound consisting of an infinite series of binomial
coefficients:
\begin{equation}
E_{\leq k}(P)  \geq 3\binom{k+2}{2}+3\binom{k+2-n/3}{2}
+3\sum_{j=2}^{\infty}j(j+1)\binom{k+2-c_{j}n}{2},
\label{eq:inftybinom}
\end{equation}
where $c_j={1}/{2}-{1}/({3j(j+1)})$.

Our main result does not improve this lower bound, however it gives
an interesting heuristic that provides some evidence about the
potential truth of this inequality for unrestricted point-sets $P$.
If we assume that the sum of the first $t+1$ terms in the right-side
of Inequality (\ref{eq:inftybinom}) is a lower bound for $E_{\leq
k}(P)$, then, just as we outlined in the previous paragraph for
$t=2$, Theorem \ref{main} gives a better bound when $k$ is big
enough. This happens to be precisely when $k \geq c_{t+1} n$, which
is also the value of $k$ for which the next term in the sum of
Inequality (\ref{eq:inftybinom}) gives a nonzero contribution.

It was also shown in \cite{ACFLS} that Inequality
(\ref{eq:inftybinom}) implies the following bound for
$3$-decomposable sets $P$:
\begin{equation}
\rcr(P) \geq
\frac{2}{27}(15-\pi^2)\tbinom{n}{4}+\Theta(n^3)>0.380029\tbinom{n}{4}+\Theta(n^3).\label{eq:cross
infty}
\end{equation}
Theorem \ref{main} does not improve the $\tbinom{n}{4}$ coefficient,
but it improves the speed of convergence. For instance, using
Theorem \ref{main} together with the first 30 terms of Inequality
(\ref{eq:inftybinom}) gives a better bound than the one obtained
solely from the first 101 terms of Inequality (\ref{eq:inftybinom}).

Finally, we reiterate our conjectures from \cite{ACFLS} that
inequalities (\ref{eq:inftybinom}) and (\ref{eq:cross infty}) are
true for unrestricted point-sets $P$. We in fact conjecture that for
every $k$ and $n$, the class of $3$-decomposable sets contains
optimal sets for both $E_{\leq k}(n)$ and $\rcr (n)$.

\end{document}